\documentclass[11pt]{article}
\usepackage{amssymb,amsmath,amsthm}
\usepackage{tikz}
\usetikzlibrary{shapes.geometric}

\newtheorem{theorem}{Theorem}[section]
\newtheorem{lemma}[theorem]{Lemma}

\newtheorem{corollary}[theorem]{Corollary}

\newtheorem{conjecture}[theorem]{Conjecture}

\title{Maximal matchings in polyspiro and benzenoid chains}
\author{Tomislav Do\v{s}li\'{c}\thanks{This work has been supported in part by Croatian Science Foundation under
the project 8481 (BioAmpMode).}\\Faculty of Civil Engineering\\University of Zagreb\\doslic@grad.hr
	\and 
	Taylor Short\thanks{This work was partially supported by a SPARC Graduate Research Grant from the Office of the Vice President for Research at the University of South Carolina and also supported in part by the NSF DMS under contract 1300547.}\\Department of Mathematics\\University of South Carolina\\shorttm2@mailbox.sc.edu}
\date{}

\begin{document}
\maketitle

\begin{abstract}
A matching $M$ of a graph $G$ is maximal if it is not a proper subset of any other
matching in $G$. Maximal matchings are much less known and researched than
their maximum and perfect counterparts. In particular, almost nothing is known about
their enumerative properties. In this paper
we present the recurrences and generating functions for the sequences 
enumerating maximal matchings in two classes of chemically interesting 
linear polymers: polyspiro 
chains and benzenoid chains. We also analyze the asymptotic behavior of those 
sequences and determine the extremal cases.

\end{abstract}

\section{Introduction}

A matching in a graph is a collection of its edges such that no two edges
in this collection have a vertex in common.
Matchings in graphs serve as successful models of many phenomena in engineering,
natural and social sciences. A strong initial impetus to their 
study came from the chemistry of benzenoid compounds after it was observed 
that the stability of benzenoid compounds is related to the existence and
the number of 
perfect matchings in the corresponding graphs. That observation gave rise to a 
number of enumerative results that were accumulated over the course of several
decades; we refer the reader to monograph \cite{cyvingut} for a survey. 
Further motivation came from
the statistical mechanics {\em via} the Kasteleyn's solution of the dimer 
problem \cite{kasteleyn1, kasteleyn2} and its applications to evaluations of partition
functions for a given value of temperature. In both cases, the matchings under
consideration are perfect, i.e., their edges are collectively incident to
all vertices of $G$. It is clear that perfect matchings are as large as
possible and that no other matching in $G$ can be ``larger'' than a perfect
one. It turns out that in all other applications we are also interested 
mostly in large matchings.

Basically, there are two ways to quantify the largeness of a matching. One way,
by using the number of edges, gives rise to the idea of {\em maximum 
matchings}. Maximum matchings are well researched and well understood; there is
a well developed structural theory and enumerative results are abundant. The
classical monograph by L\'ovasz and Plummer \cite{lovasz} is an excellent
reference for all aspects of the theory.

An alternative way is to say that a matching is large if no other matching
contains it as a proper subset; this gives rise to the concept of 
{\em maximal matchings}. Every maximum matching is also maximal, but the
opposite is usually not true. Unlike their maximum counterparts, maximal 
matchings can have different cardinalities (unless the graph is 
{\em equimatchable}; see \cite{frendrup}) and the recurrences used for
their enumeration are essentially non-local. As a consequence, maximal matchings
are much less understood then the maximum ones. There is nothing analogous 
to the structural theory of maximum matchings and the enumerative results 
are scarce and scattered through the literature \cite{doszub2,klazar,wagner}.

In spite of their obscurity, maximal matchings are natural models for several
problems connected with adsorption of dimers on a structured substrate and
block-allocation of a sequential resource. One can find them also in the 
context of polymerization of organic molecules, as witnessed by an early 
paper of Flory \cite{flory}. A probabilistic approach to the same problem
can be found in \cite{montroll}. We refer the reader to papers 
\cite{andova,doslicsat,doszub1,doszub2} for some structural and enumerative 
results on those models. In this paper our goal is to further the line of 
research of reference \cite{doszub2} by considering graphs with more 
complicated connectivity patterns and richer structure of basic units. 
We provide enumerative and extremal results on maximal matchings
in two classes of linear polymers of chemical interest: the polyspiro chains
and benzenoid chains. We extablish the recurrences and generating
functions for the enumerating sequences of maximal matchings in three classes
of uniform polyspiro chains and use the obtained results to determine the
asymptotic behavior and to find the extremal chains. Further, we also
enumerate maximal matchings in three classes of benzenoid chains and show
that one of them is extremal with respect to the number of maximal matchings.
Our results show that maximal matchings behave in a radically different
way that the perfect matchings; chains rich in maximal matchings are poor
in perfect matchings and vice versa. We end by comparing our results with 
enumerative results for other type of structures in similar polymers and
by discussing some possible directions of future research.

\section{Preliminaries}

Our terminology and notations are mostly standard and taken from
\cite{lovasz,west}. All graphs $G$ considered in this paper will be finite and simple, with vertex set $V(G)$ and edge set $E(G)$. For a subset of vertices $S$ of $V(G)$, we make use of the notation $G - S$ (or $G-v$ if $S=\{v\}$) to denote the subgraph of $G$ obtained by deleting the vertices of $S$ and all edges incident to them. For a graph $G$ and subset of edges $X$ of $G$, we use the notation $G\setminus X$ (or $G\setminus e$ if $X=\{e\}$) to denote the subgraph of $G$ obtained by deleting the endpoints of the edges in $X$ as well as all incident edges to these endpoints. 

A \emph{matching} $M$ in G is a set of edges of $G$ such that no two edges from $M$ have a vertex in common. The number of edges in $M$ is called its \emph{size}. A matching in $G$ with the largest possible size is called a \emph{maximum matching}. If a matching in $G$ is not a subset of a larger matching of $G$, it is called a \emph{maximal matching}. Let $\Psi(G)$ denote the number of maximal matchings of $G$.

In this paper we are mainly concerned with counting maximal matchings in two classes of linear polymers (or {\em facsiagraphs}, \cite{juvan}) with simple connectivity patterns. The first class are $6$-uniform cactus chains. Chain cacti are in chemical literature known as {\em polyspiro chains}.

A \emph{cactus graph} is a connected graph in which no edge lies in more than one cycle. Consequently, each block of a cactus graph is either an edge or a cycle. If all blocks of a cactus $G$ are cycles of the same length $m$, the cactus is \emph{$m$-uniform}. 

A \emph{hexagonal cactus} is a 6-uniform cactus, i.e., a cactus in which every block is a hexagon. A vertex shared by two or more hexagons is called a \emph{cut-vertex}. If each hexagon of a hexagonal cactus $G$ has at most two cut-vertices, and each cut-vertex is shared by exactly two hexagons, we say that $G$ is a \emph{chain hexagonal cactus}. The number of hexagons is called the \emph{length} of the chain. An example of a chain hexagonal cactus is shown in Figure \ref{excactus}. 

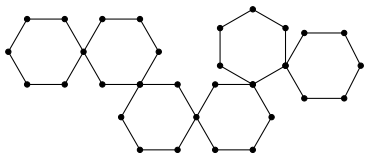
\begin{figure}[h]
\begin{center}
\begin{tikzpicture}[scale=.5]
\draw (-.5,.87) circle (2pt) [fill=black];
\draw (.5,.87) circle (2pt) [fill=black];
\draw (.5,-.87) circle (2pt) [fill=black];
\draw (-.5,-.87) circle (2pt) [fill=black];
\draw (1,0) circle (2pt) [fill=black];
\draw (-1,0) circle (2pt) [fill=black];
\draw (1,0)--(.5,.87);
\draw (-.5,.87)--(.5,.87);
\draw (-.5,.87)--(-1,0);
\draw (-.5,-.87)--(-1,0);
\draw (-.5,-.87)--(.5,-.87);
\draw (1,0)--(.5,-.87);

\draw (1.5,.87) circle (2pt) [fill=black];
\draw (2.5,.87) circle (2pt) [fill=black];
\draw (2.5,-.87) circle (2pt) [fill=black];
\draw (1.5,-.87) circle (2pt) [fill=black];
\draw (3,0) circle (2pt) [fill=black];
\draw (3,0)--(2.5,.87);
\draw (1.5,.87)--(2.5,.87);
\draw (1.5,.87)--(1,0);
\draw (1.5,-.87)--(1,0);
\draw (1.5,-.87)--(2.5,-.87);
\draw (3,0)--(2.5,-.87);

\draw (2.5,-.87) circle (2pt) [fill=black];
\draw (3.5,-.87) circle (2pt) [fill=black];
\draw (3.5,-2.61) circle (2pt) [fill=black];
\draw (2.5,-2.61) circle (2pt) [fill=black];
\draw (4,-1.74) circle (2pt) [fill=black];
\draw (2,-1.74) circle (2pt) [fill=black];
\draw (4,-1.74)--(3.5,-.87);
\draw (2.5,-.87)--(3.5,-.87);
\draw (2.5,-.87)--(2,-1.74);
\draw (2.5,-2.61)--(2,-1.74);
\draw (2.5,-2.61)--(3.5,-2.61);
\draw (4,-1.74)--(3.5,-2.61);

\draw (4.5,-.87) circle (2pt) [fill=black];
\draw (5.5,-.87) circle (2pt) [fill=black];
\draw (5.5,-2.61) circle (2pt) [fill=black];
\draw (4.5,-2.61) circle (2pt) [fill=black];
\draw (6,-1.74) circle (2pt) [fill=black];
\draw (6,-1.74)--(5.5,-.87);
\draw (4.5,-.87)--(5.5,-.87);
\draw (4.5,-.87)--(4,-1.74);
\draw (4.5,-2.61)--(4,-1.74);
\draw (4.5,-2.61)--(5.5,-2.61);
\draw (6,-1.74)--(5.5,-2.61);

\draw (5.5,-.87) circle (2pt) [fill=black];
\draw (5.5,1.13) circle (2pt) [fill=black];
\draw (4.63,-.37) circle (2pt) [fill=black];
\draw (4.63,.63) circle (2pt) [fill=black];
\draw (6.37,-.37) circle (2pt) [fill=black];
\draw (6.37,.63) circle (2pt) [fill=black];
\draw (5.5,-.87)--(4.63,-.37);
\draw (5.5,-.87)--(6.37,-.37);
\draw (6.37,-.37)--(6.37,.63);
\draw (6.37,.63)--(5.5,1.13);
\draw (5.5,1.13)--(4.63,.63);
\draw (4.63,.63)--(4.63,-.37);

\draw (6.37,-.37) circle (2pt) [fill=black];
\draw (8.37,-.37) circle (2pt) [fill=black];
\draw (6.87,.5) circle (2pt) [fill=black];
\draw (7.93,.5) circle (2pt) [fill=black];
\draw (6.87,-1.24) circle (2pt) [fill=black];
\draw (7.93,-1.24) circle (2pt) [fill=black];
\draw (6.37,-.37)--(6.87,.5);
\draw (6.87,.5)--(7.93,.5);
\draw (7.93,.5)--(8.37,-.37);
\draw (8.37,-.37)--(7.93,-1.24);
\draw (7.93,-1.24)--(6.87,-1.24);
\draw (6.87,-1.24)--(6.37,-.37);

\end{tikzpicture}
\end{center}
\caption{A chain hexagonal cactus of length 6.} \label{excactus}
\end{figure}

Furthermore, any chain hexagonal cactus of length greater than one has exactly two hexagons with only one cut-vertex; such hexagons are called \emph{terminal} and all other hexagons with two cut-vertices are called \emph{internal}. 

Internal hexagons can be one of three types depending upon the distance between its cut-vertices: in an \emph{ortho-hexagon} cut vertices are adjacent, in a \emph{meta-hexagon} they are at distance two, and in a \emph{para-hexagon} cut-vertices are at distance three. The terminology is borrowed from the theory of benzenoid hydrocarbons; see \cite{doslicmaloy,doszub1,doszub2} for more details. These give rise to the following three types of hexagonal cactus chains of length $n$: the unique chain whose internal hexagons are all para-hexagons is $P_n$, the unique chain whose internal hexagons are all meta-hexagons is $M_n$, and the unique chain whose internal hexagons are all ortho-hexagons is $O_n$.  

\begin{figure}[h]
\begin{center}
\begin{tikzpicture}[scale=.5]
\draw (-2.5,0) node {$P_n$};
\draw (-.5,.87) circle (2pt) [fill=black];
\draw (.5,.87) circle (2pt) [fill=black];
\draw (.5,-.87) circle (2pt) [fill=black];
\draw (-.5,-.87) circle (2pt) [fill=black];
\draw (1,0) circle (2pt) [fill=black];
\draw (-1,0) circle (2pt) [fill=black];
\draw (1,0)--(.5,.87);
\draw (-.5,.87)--(.5,.87);
\draw (-.5,.87)--(-1,0);
\draw (-.5,-.87)--(-1,0);
\draw (-.5,-.87)--(.5,-.87);
\draw (1,0)--(.5,-.87);

\draw (1.5,.87) circle (2pt) [fill=black];
\draw (2.5,.87) circle (2pt) [fill=black];
\draw (2.5,-.87) circle (2pt) [fill=black];
\draw (1.5,-.87) circle (2pt) [fill=black];
\draw (3,0) circle (2pt) [fill=black];
\draw (3,0)--(2.5,.87);
\draw (1.5,.87)--(2.5,.87);
\draw (1.5,.87)--(1,0);
\draw (1.5,-.87)--(1,0);
\draw (1.5,-.87)--(2.5,-.87);
\draw (3,0)--(2.5,-.87);

\draw (3.5,.87) circle (2pt) [fill=black];
\draw (4.5,.87) circle (2pt) [fill=black];
\draw (4.5,-.87) circle (2pt) [fill=black];
\draw (3.5,-.87) circle (2pt) [fill=black];
\draw (5,0) circle (2pt) [fill=black];
\draw (5,0)--(4.5,.87);
\draw (3.5,.87)--(4.5,.87);
\draw (3.5,.87)--(3,0);
\draw (3.5,-.87)--(3,0);
\draw (3.5,-.87)--(4.5,-.87);
\draw (5,0)--(4.5,-.87);

\draw (5.5,.87) circle (2pt) [fill=black];
\draw (6.5,.87) circle (2pt) [fill=black];
\draw (6.5,-.87) circle (2pt) [fill=black];
\draw (5.5,-.87) circle (2pt) [fill=black];
\draw (7,0) circle (2pt) [fill=black];
\draw (7,0)--(6.5,.87);
\draw (5.5,.87)--(6.5,.87);
\draw (5.5,.87)--(5,0);
\draw (5.5,-.87)--(5,0);
\draw (5.5,-.87)--(6.5,-.87);
\draw (7,0)--(6.5,-.87);

\draw (7.5,.87) circle (2pt) [fill=black];
\draw (8.5,.87) circle (2pt) [fill=black];
\draw (8.5,-.87) circle (2pt) [fill=black];
\draw (7.5,-.87) circle (2pt) [fill=black];
\draw (9,0) circle (2pt) [fill=black];
\draw (9,0)--(8.5,.87);
\draw (7.5,.87)--(8.5,.87);
\draw (7.5,.87)--(7,0);
\draw (7.5,-.87)--(7,0);
\draw (7.5,-.87)--(8.5,-.87);
\draw (9,0)--(8.5,-.87);

\draw (0,0) node {$1$};
\draw (2,0) node {$2$};
\draw (4,0) node {$\cdots$};
\draw (8,0) node {$n$};
\end{tikzpicture}

\vspace{.4cm}

\begin{tikzpicture}[scale=.5]
\draw (-2.5,0) node {$M_n$};
\draw (.87,-.5) circle (2pt) [fill=black];
\draw (.87,.5) circle (2pt) [fill=black];
\draw (-.87,.5) circle (2pt) [fill=black];
\draw (-.87,-.5) circle (2pt) [fill=black];
\draw (0,1) circle (2pt) [fill=black];
\draw (0,-1) circle (2pt) [fill=black];
\draw (0,1)--(.87,.5);
\draw (.87,-.5)--(.87,.5);
\draw (.87,-.5)--(0,-1);
\draw (-.87,-.5)--(0,-1);
\draw (-.87,-.5)--(-.87,.5);
\draw (0,1)--(-.87,.5);

\draw (2.61,-.5) circle (2pt) [fill=black];
\draw (.87,-1.5) circle (2pt) [fill=black];
\draw (2.61,-1.5) circle (2pt) [fill=black];
\draw (1.74,0) circle (2pt) [fill=black];
\draw (1.74,-2) circle (2pt) [fill=black];
\draw (.87,-1.5)--(.87,-.5);
\draw (2.61,-1.5)--(1.74,-2);
\draw (2.61,-1.5)--(2.61,-.5);
\draw (1.74,0)--(2.61,-.5);
\draw (1.74,0)--(.87,-.5);
\draw (1.74,-2)--(.87,-1.5);

\draw (4.35,.5) circle (2pt) [fill=black];
\draw (4.35,-.5) circle (2pt) [fill=black];
\draw (3.48,1) circle (2pt) [fill=black];
\draw (3.48,-1) circle (2pt) [fill=black];
\draw (2.61,.5) circle (2pt) [fill=black];
\draw (2.61,.5)--(2.61,-1);
\draw (4.35,-.5)--(3.48,-1);
\draw (4.35,-.5)--(4.35,.5);
\draw (3.48,1)--(4.35,.5);
\draw (3.48,1)--(2.61,.5);
\draw (3.48,-1)--(2.61,-.5);

\draw (6.09,-.5) circle (2pt) [fill=black];
\draw (6.09,-1.5) circle (2pt) [fill=black];
\draw (5.22,0) circle (2pt) [fill=black];
\draw (5.22,-2) circle (2pt) [fill=black];
\draw (4.35,-1.5) circle (2pt) [fill=black];
\draw (4.35,-1.5)--(4.35,-.5);
\draw (6.09,-1.5)--(5.22,-2);
\draw (6.09,-1.5)--(6.09,-.5);
\draw (5.22,0)--(6.09,-.5);
\draw (5.22,0)--(4.35,-.5);
\draw (5.22,-2)--(4.35,-1.5);

\draw (7.83,.5) circle (2pt) [fill=black];
\draw (7.83,-.5) circle (2pt) [fill=black];
\draw (6.96,1) circle (2pt) [fill=black];
\draw (6.96,-1) circle (2pt) [fill=black];
\draw (6.09,.5) circle (2pt) [fill=black];
\draw (6.09,.5)--(6.09,-.5);
\draw (7.83,-.5)--(6.96,-1);
\draw (7.83,-.5)--(7.83,.5);
\draw (6.96,1)--(7.83,.5);
\draw (6.96,1)--(6.09,.5);
\draw (6.96,-1)--(6.09,-.5);

\draw (0,0) node {$1$};
\draw (1.74,-1) node {$2$};
\draw (3.48,0) node {$\cdots$};
\draw (6.96,0) node {$n$};
\end{tikzpicture}

\vspace{.4cm}

\begin{tikzpicture}[scale=.5]
\draw (-2.5,0) node {$O_n$};
\draw (-.7,.87) circle (2pt) [fill=black];
\draw (.7,.87) circle (2pt) [fill=black];
\draw (.7,-.87) circle (2pt) [fill=black];
\draw (-.7,-.87) circle (2pt) [fill=black];
\draw (1,0) circle (2pt) [fill=black];
\draw (-1,0) circle (2pt) [fill=black];
\draw (1,0)--(.7,.87);
\draw (-.7,.87)--(.7,.87);
\draw (-.7,.87)--(-1,0);
\draw (-.7,-.87)--(-1,0);
\draw (-.7,-.87)--(.7,-.87);
\draw (1,0)--(.7,-.87);

\draw (.7,-.87) circle (2pt) [fill=black];
\draw (2.1,-.87) circle (2pt) [fill=black];
\draw (2.1,-2.61) circle (2pt) [fill=black];
\draw (.7,-2.61) circle (2pt) [fill=black];
\draw (2.4,-1.74) circle (2pt) [fill=black];
\draw (.4,-1.74) circle (2pt) [fill=black];
\draw (2.4,-1.74)--(2.1,-.87);
\draw (.7,-.87)--(2.1,-.87);
\draw (.7,-.87)--(.4,-1.74);
\draw (.7,-2.61)--(.4,-1.74);
\draw (.7,-2.61)--(2.1,-2.61);
\draw (2.4,-1.74)--(2.1,-2.61);

\draw (2.1,.87) circle (2pt) [fill=black];
\draw (3.5,.87) circle (2pt) [fill=black];
\draw (3.5,-.87) circle (2pt) [fill=black];
\draw (2.1,-.87) circle (2pt) [fill=black];
\draw (3.8,0) circle (2pt) [fill=black];
\draw (1.8,0) circle (2pt) [fill=black];
\draw (3.8,0)--(3.5,.87);
\draw (2.1,.87)--(3.5,.87);
\draw (2.1,.87)--(1.8,0);
\draw (2.1,-.87)--(1.8,0);
\draw (2.1,-.87)--(3.5,-.87);
\draw (3.8,0)--(3.5,-.87);

\draw (3.5,-.87) circle (2pt) [fill=black];
\draw (4.9,-.87) circle (2pt) [fill=black];
\draw (4.9,-2.61) circle (2pt) [fill=black];
\draw (3.5,-2.61) circle (2pt) [fill=black];
\draw (5.2,-1.74) circle (2pt) [fill=black];
\draw (3.2,-1.74) circle (2pt) [fill=black];
\draw (5.2,-1.74)--(4.9,-.87);
\draw (3.5,-.87)--(4.9,-.87);
\draw (3.5,-.87)--(3.2,-1.74);
\draw (3.5,-2.61)--(3.2,-1.74);
\draw (3.5,-2.61)--(4.9,-2.61);
\draw (5.2,-1.74)--(4.9,-2.61);

\draw (4.9,.87) circle (2pt) [fill=black];
\draw (6.3,.87) circle (2pt) [fill=black];
\draw (6.3,-.87) circle (2pt) [fill=black];
\draw (4.9,-.87) circle (2pt) [fill=black];
\draw (6.6,0) circle (2pt) [fill=black];
\draw (4.6,0) circle (2pt) [fill=black];
\draw (6.6,0)--(6.3,.87);
\draw (4.9,.87)--(6.3,.87);
\draw (4.9,.87)--(4.6,0);
\draw (4.9,-.87)--(4.6,0);
\draw (4.9,-.87)--(6.3,-.87);
\draw (6.6,0)--(6.3,-.87);

\draw (0,0) node {$1$};
\draw (1.4,-1.74) node {$2$};
\draw (2.8,0) node {$\cdots$};
\draw (5.6,0) node {$n$};
\end{tikzpicture}

\end{center}
\caption{The hexagonal cactus chains $P_n$, $M_n$, and $O_n$.} \label{}
\end{figure}
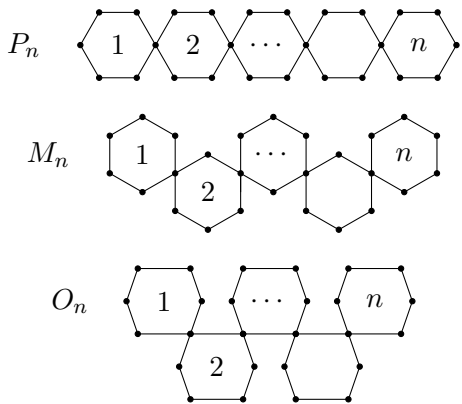

The other class of unbranched polymers we consider are benzenoid chains. A \emph{benzenoid system} is a is a connected, plane graph without cut-vertices in which all faces, except the unbounded one, are hexagons. Two hexagonal faces are either disjoint or they share exactly one common edge (adjacent hexagons). A vertex of a benzenoid system belongs to at most three hexagonal faces and the benzenoid system is called \emph{catacondensed} if it does not posses such a vertex. If no hexagon in a catacondensed benzenoid is adjacent to three other hexagons, we say that the benzenoid is a \emph{chain} see Figure \ref{exbenz1}.

The number of hexagons in a benzenoid chain is called its \emph{length}. In each benzenoid chain there are exactly two hexagons adjacent to one other hexagon; those two hexagons are called \emph{terminal}, while any other hexagons are called \emph{interior}. An interior hexagon has two vertices of degree 2. If these two vertices are not adjacent, then hexagon is called \emph{straight}. If the two vertices are adjacent, then the hexagon is called \emph{kinky}.

\begin{figure}[h]
\begin{center}
\begin{tikzpicture}[scale=.5]
\draw (-.5,.87) circle (2pt) [fill=black];
\draw (.5,.87) circle (2pt) [fill=black];
\draw (.5,-.87) circle (2pt) [fill=black];
\draw (-.5,-.87) circle (2pt) [fill=black];
\draw (1,0) circle (2pt) [fill=black];
\draw (-1,0) circle (2pt) [fill=black];
\draw (1,0)--(.5,.87);
\draw (-.5,.87)--(.5,.87);
\draw (-.5,.87)--(-1,0);
\draw (-.5,-.87)--(-1,0);
\draw (-.5,-.87)--(.5,-.87);
\draw (1,0)--(.5,-.87);

\draw (2,0) circle (2pt) [fill=black];
\draw (2.5,-.87) circle (2pt) [fill=black];
\draw (1,-1.74) circle (2pt) [fill=black];
\draw (2,-1.74) circle (2pt) [fill=black];
\draw (1,0)--(2,0);
\draw (2.5,-.87)--(2,0);
\draw (2.5,-.87)--(2,-1.74);
\draw (1,-1.74)--(2,-1.74);
\draw (1,-1.74)--(.5,-.87);

\draw (2.5,.87) circle (2pt) [fill=black];
\draw (3.5,.87) circle (2pt) [fill=black];
\draw (3.5,-.87) circle (2pt) [fill=black];
\draw (4,0) circle (2pt) [fill=black];
\draw (4,0)--(3.5,.87);
\draw (2.5,.87)--(3.5,.87);
\draw (2.5,.87)--(2,0);
\draw (2.5,-.87)--(3.5,-.87);
\draw (4,0)--(3.5,-.87);

\draw (5,1.74) circle (2pt) [fill=black];
\draw (5.5,.87) circle (2pt) [fill=black];
\draw (4,1.74) circle (2pt) [fill=black];
\draw (5,0) circle (2pt) [fill=black];
\draw (4,1.74)--(5,1.74);
\draw (5.5,.87)--(5,1.74);
\draw (5.5,.87)--(5,0);
\draw (4,0)--(5,0);
\draw (4,1.74)--(3.5,.87);

\draw (5.5,2.61) circle (2pt) [fill=black];
\draw (6.5,2.61) circle (2pt) [fill=black];
\draw (6.5,.87) circle (2pt) [fill=black];
\draw (7,1.74) circle (2pt) [fill=black];
\draw (7,1.74)--(6.5,2.61);
\draw (5.5,2.61)--(6.5,2.61);
\draw (5.5,2.61)--(5,1.74);
\draw (5.5,.87)--(6.5,.87);
\draw (7,1.74)--(6.5,.87);

\draw (8,1.74) circle (2pt) [fill=black];
\draw (8.5,.87) circle (2pt) [fill=black];
\draw (7,1.74) circle (2pt) [fill=black];
\draw (8,0) circle (2pt) [fill=black];
\draw (7,0) circle (2pt) [fill=black];
\draw (7,1.74)--(8,1.74);
\draw (8.5,.87)--(8,1.74);
\draw (8.5,.87)--(8,0);
\draw (7,0)--(8,0);
\draw (7,0)--(6.5,.87);
\end{tikzpicture}

\end{center}
\caption{A benzenoid chain of length 6.} \label{exbenz1}
\end{figure}
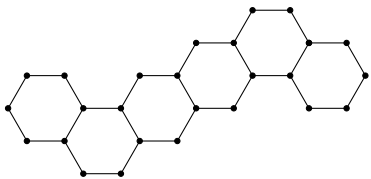

If all $n-2$ interior hexagons of a benzenoid chain with $n$ hexagons are straight, we call the chain a \emph{polyacene} and denote it by $L_n$. If all interior hexagons are kinky, the chain is called a \emph{polyphenacene}. Since the number of perfect matchings in a polyphenacene of length $n$ is equal to the $(n + 2)$-nd Fibonacci number $F_{n+2}$, these chains are also known as \emph{fibonacenes} \cite{cyvingut}. We consider two specific families of polyphenacenes depicted in Figure \ref{exbenz2}: the \emph{zig-zag polyphenacene}, $Z_n$, and \emph{helicene}, $H_n$.

\begin{figure}[h]
\begin{center}
\begin{tikzpicture}[scale=.5]
\draw (-2.5,0) node {$L_n$};
\draw (.87,-.5) circle (2pt) [fill=black];
\draw (.87,.5) circle (2pt) [fill=black];
\draw (-.87,.5) circle (2pt) [fill=black];
\draw (-.87,-.5) circle (2pt) [fill=black];
\draw (0,1) circle (2pt) [fill=black];
\draw (0,-1) circle (2pt) [fill=black];
\draw (0,1)--(.87,.5);
\draw (.87,-.5)--(.87,.5);
\draw (.87,-.5)--(0,-1);
\draw (-.87,-.5)--(0,-1);
\draw (-.87,-.5)--(-.87,.5);
\draw (0,1)--(-.87,.5);

\draw (2.61,.5) circle (2pt) [fill=black];
\draw (2.61,-.5) circle (2pt) [fill=black];
\draw (1.74,1) circle (2pt) [fill=black];
\draw (1.74,-1) circle (2pt) [fill=black];
\draw (2.61,-.5)--(1.74,-1);
\draw (2.61,-.5)--(2.61,.5);
\draw (1.74,1)--(2.61,.5);
\draw (1.74,1)--(.87,.5);
\draw (1.74,-1)--(.87,-.5);

\draw (4.35,.5) circle (2pt) [fill=black];
\draw (4.35,-.5) circle (2pt) [fill=black];
\draw (3.48,1) circle (2pt) [fill=black];
\draw (3.48,-1) circle (2pt) [fill=black];
\draw (4.35,-.5)--(3.48,-1);
\draw (4.35,-.5)--(4.35,.5);
\draw (3.48,1)--(4.35,.5);
\draw (3.48,1)--(2.61,.5);
\draw (3.48,-1)--(2.61,-.5);

\draw (6.09,.5) circle (2pt) [fill=black];
\draw (6.09,-.5) circle (2pt) [fill=black];
\draw (5.22,1) circle (2pt) [fill=black];
\draw (5.22,-1) circle (2pt) [fill=black];
\draw (6.09,-.5)--(5.22,-1);
\draw (6.09,-.5)--(6.09,.5);
\draw (5.22,1)--(6.09,.5);
\draw (5.22,1)--(4.35,.5);
\draw (5.22,-1)--(4.35,-.5);

\draw (7.83,.5) circle (2pt) [fill=black];
\draw (7.83,-.5) circle (2pt) [fill=black];
\draw (6.96,1) circle (2pt) [fill=black];
\draw (6.96,-1) circle (2pt) [fill=black];
\draw (7.83,-.5)--(6.96,-1);
\draw (7.83,-.5)--(7.83,.5);
\draw (6.96,1)--(7.83,.5);
\draw (6.96,1)--(6.09,.5);
\draw (6.96,-1)--(6.09,-.5);

\draw (0,0) node {$1$};
\draw (1.74,0) node {$2$};
\draw (3.48,0) node {$\cdots$};
\draw (6.96,0) node {$n$};
\end{tikzpicture}

\vspace{.4cm}

\begin{tikzpicture}[scale=.5]
\draw (-2.5,0) node {$Z_n$};
\draw (-.5,.87) circle (2pt) [fill=black];
\draw (.5,.87) circle (2pt) [fill=black];
\draw (.5,-.87) circle (2pt) [fill=black];
\draw (-.5,-.87) circle (2pt) [fill=black];
\draw (1,0) circle (2pt) [fill=black];
\draw (-1,0) circle (2pt) [fill=black];
\draw (1,0)--(.5,.87);
\draw (-.5,.87)--(.5,.87);
\draw (-.5,.87)--(-1,0);
\draw (-.5,-.87)--(-1,0);
\draw (-.5,-.87)--(.5,-.87);
\draw (1,0)--(.5,-.87);

\draw (2,0) circle (2pt) [fill=black];
\draw (2.5,-.87) circle (2pt) [fill=black];
\draw (1,-1.74) circle (2pt) [fill=black];
\draw (2,-1.74) circle (2pt) [fill=black];
\draw (1,0)--(2,0);
\draw (2.5,-.87)--(2,0);
\draw (2.5,-.87)--(2,-1.74);
\draw (1,-1.74)--(2,-1.74);
\draw (1,-1.74)--(.5,-.87);

\draw (2.5,.87) circle (2pt) [fill=black];
\draw (3.5,.87) circle (2pt) [fill=black];
\draw (3.5,-.87) circle (2pt) [fill=black];
\draw (4,0) circle (2pt) [fill=black];
\draw (4,0)--(3.5,.87);
\draw (2.5,.87)--(3.5,.87);
\draw (2.5,.87)--(2,0);
\draw (2.5,-.87)--(3.5,-.87);
\draw (4,0)--(3.5,-.87);

\draw (5,0) circle (2pt) [fill=black];
\draw (5.5,-.87) circle (2pt) [fill=black];
\draw (4,-1.74) circle (2pt) [fill=black];
\draw (5,-1.74) circle (2pt) [fill=black];
\draw (4,0)--(5,0);
\draw (5.5,-.87)--(5,0);
\draw (5.5,-.87)--(5,-1.74);
\draw (4,-1.74)--(5,-1.74);
\draw (4,-1.74)--(3.5,-.87);

\draw (5.5,.87) circle (2pt) [fill=black];
\draw (6.5,.87) circle (2pt) [fill=black];
\draw (6.5,-.87) circle (2pt) [fill=black];
\draw (7,0) circle (2pt) [fill=black];
\draw (7,0)--(6.5,.87);
\draw (5.5,.87)--(6.5,.87);
\draw (5.5,.87)--(5,0);
\draw (5.5,-.87)--(6.5,-.87);
\draw (7,0)--(6.5,-.87);

\draw (0,0) node {$1$};
\draw (1.5,-.87) node {$2$};
\draw (3,0) node {$\cdots$};
\draw (6,0) node {$n$};
\end{tikzpicture}

\vspace{.4cm}

\begin{tikzpicture}[scale=.5]
\draw (-2.5,0) node {$H_n$};
\draw (.87,-.5) circle (2pt) [fill=black];
\draw (.87,.5) circle (2pt) [fill=black];
\draw (-.87,.5) circle (2pt) [fill=black];
\draw (-.87,-.5) circle (2pt) [fill=black];
\draw (0,1) circle (2pt) [fill=black];
\draw (0,-1) circle (2pt) [fill=black];
\draw (0,1)--(.87,.5);
\draw (.87,-.5)--(.87,.5);
\draw (.87,-.5)--(0,-1);
\draw (-.87,-.5)--(0,-1);
\draw (-.87,-.5)--(-.87,.5);
\draw (0,1)--(-.87,.5);

\draw (2.61,.5) circle (2pt) [fill=black];
\draw (2.61,-.5) circle (2pt) [fill=black];
\draw (1.74,1) circle (2pt) [fill=black];
\draw (1.74,-1) circle (2pt) [fill=black];
\draw (2.61,-.5)--(1.74,-1);
\draw (2.61,-.5)--(2.61,.5);
\draw (1.74,1)--(2.61,.5);
\draw (1.74,1)--(.87,.5);
\draw (1.74,-1)--(.87,-.5);

\draw (3.48,-1) circle (2pt) [fill=black];
\draw (3.48,-2) circle (2pt) [fill=black];
\draw (2.61,-2.5) circle (2pt) [fill=black];
\draw (1.74,-2) circle (2pt) [fill=black];
\draw (1.74,-2)--(1.74,-1);
\draw (3.48,-2)--(2.61,-2.5);
\draw (3.48,-2)--(3.48,-1);
\draw (2.61,-.5)--(3.48,-1);
\draw (2.61,-2.5)--(1.74,-2);

\draw (2.61,-3.5) circle (2pt) [fill=black];
\draw (1.74,-4) circle (2pt) [fill=black];
\draw (.87,-2.5) circle (2pt) [fill=black];
\draw (.87,-3.5) circle (2pt) [fill=black];
\draw (.87,-2.5)--(.87,-3.5);
\draw (2.61,-3.5)--(1.74,-4);
\draw (2.61,-3.5)--(2.61,-2.5);
\draw (1.74,-2)--(.87,-2.5);
\draw (1.74,-4)--(.87,-3.5);

\draw (-.87,-2.5) circle (2pt) [fill=black];
\draw (-.87,-3.5) circle (2pt) [fill=black];
\draw (0,-2) circle (2pt) [fill=black];
\draw (0,-4) circle (2pt) [fill=black];
\draw (0,-2)--(.87,-2.5);
\draw (.87,-3.5)--(0,-4);
\draw (-.87,-3.5)--(0,-4);
\draw (-.87,-3.5)--(-.87,-2.5);
\draw (0,-2)--(-.87,-2.5);

\draw (0,0) node {$1$};
\draw (1.74,0) node {$2$};
\draw (2.61,-1.4) node {$\vdots$};
\draw (0,-3) node {$n$};
\end{tikzpicture}

\end{center}
\caption{The polyacene, zig-zag polyphenacene, and helicene chains.} \label{exbenz2}
\end{figure}
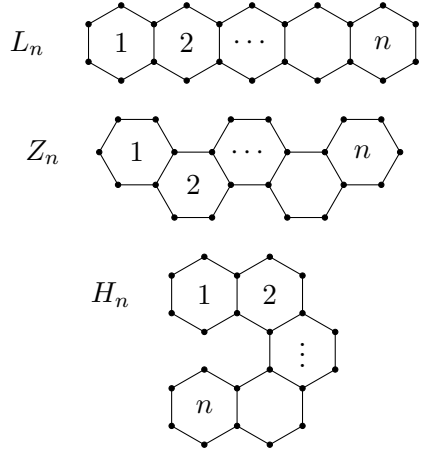

\section{Chain hexagonal cacti}

\subsection{Generating functions} \label{seccactigen}

In this section, we obtain ordinary generating functions for the number of maximal matchings in the hexagonal chain cacti $P_n$, $M_n$, and $O_n$. To do this, we first find recursions for the number of maximal matchings using auxiliary graphs (initial conditions are obtained by direct counting). These recursions can be verified via casework. By introducing generating functions for the number of maximal matchings in each auxiliary graph, the recursions can be transformed into a solvable system of equations in terms of unknown generating functions. Finally, we solve this system of equations for the desired generating function. We omit the details of most of these computations.

\begin{figure}[h]
\begin{center}
\begin{tikzpicture}[scale=.4]
\draw (-2.5,0) node {$P_n$};
\draw (-.5,.87) circle (2pt) [fill=black];
\draw (.5,.87) circle (2pt) [fill=black];
\draw (.5,-.87) circle (2pt) [fill=black];
\draw (-.5,-.87) circle (2pt) [fill=black];
\draw (1,0) circle (2pt) [fill=black];
\draw (-1,0) circle (2pt) [fill=black];
\draw (1,0)--(.5,.87);
\draw (-.5,.87)--(.5,.87);
\draw (-.5,.87)--(-1,0);
\draw (-.5,-.87)--(-1,0);
\draw (-.5,-.87)--(.5,-.87);
\draw (1,0)--(.5,-.87);

\draw (1.5,.87) circle (2pt) [fill=black];
\draw (2.5,.87) circle (2pt) [fill=black];
\draw (2.5,-.87) circle (2pt) [fill=black];
\draw (1.5,-.87) circle (2pt) [fill=black];
\draw (3,0) circle (2pt) [fill=black];
\draw (3,0)--(2.5,.87);
\draw (1.5,.87)--(2.5,.87);
\draw (1.5,.87)--(1,0);
\draw (1.5,-.87)--(1,0);
\draw (1.5,-.87)--(2.5,-.87);
\draw (3,0)--(2.5,-.87);

\draw (3.5,.87) circle (2pt) [fill=black];
\draw (4.5,.87) circle (2pt) [fill=black];
\draw (4.5,-.87) circle (2pt) [fill=black];
\draw (3.5,-.87) circle (2pt) [fill=black];
\draw (5,0) circle (2pt) [fill=black];
\draw (5,0)--(4.5,.87);
\draw (3.5,.87)--(4.5,.87);
\draw (3.5,.87)--(3,0);
\draw (3.5,-.87)--(3,0);
\draw (3.5,-.87)--(4.5,-.87);
\draw (5,0)--(4.5,-.87);

\draw (5.5,.87) circle (2pt) [fill=black];
\draw (6.5,.87) circle (2pt) [fill=black];
\draw (6.5,-.87) circle (2pt) [fill=black];
\draw (5.5,-.87) circle (2pt) [fill=black];
\draw (7,0) circle (2pt) [fill=black];
\draw (7,0)--(6.5,.87);
\draw (5.5,.87)--(6.5,.87);
\draw (5.5,.87)--(5,0);
\draw (5.5,-.87)--(5,0);
\draw (5.5,-.87)--(6.5,-.87);
\draw (7,0)--(6.5,-.87);

\draw (7.5,.87) circle (2pt) [fill=black];
\draw (8.5,.87) circle (2pt) [fill=black];
\draw (8.5,-.87) circle (2pt) [fill=black];
\draw (7.5,-.87) circle (2pt) [fill=black];
\draw (9,0) circle (2pt) [fill=black];
\draw (9,0)--(8.5,.87);
\draw (7.5,.87)--(8.5,.87);
\draw (7.5,.87)--(7,0);
\draw (7.5,-.87)--(7,0);
\draw (7.5,-.87)--(8.5,-.87);
\draw (9,0)--(8.5,-.87);

\draw (0,0) node {$1$};
\draw (2,0) node {$2$};
\draw (4,0) node {$\cdots$};
\draw (8,0) node {$n$};
\end{tikzpicture}
\hspace{.5cm}
\begin{tikzpicture}[scale=.4]
\draw (-2.5,0) node {$P^1_n$};
\draw (-.5,.87) circle (2pt) [fill=black];
\draw (.5,.87) circle (2pt) [fill=black];
\draw (.5,-.87) circle (2pt) [fill=black];
\draw (-.5,-.87) circle (2pt) [fill=black];
\draw (1,0) circle (2pt) [fill=black];
\draw (-1,0) circle (2pt) [fill=black];
\draw (1,0)--(.5,.87);
\draw (-.5,.87)--(.5,.87);
\draw (-.5,.87)--(-1,0);
\draw (-.5,-.87)--(-1,0);
\draw (-.5,-.87)--(.5,-.87);
\draw (1,0)--(.5,-.87);

\draw (1.5,.87) circle (2pt) [fill=black];
\draw (2.5,.87) circle (2pt) [fill=black];
\draw (2.5,-.87) circle (2pt) [fill=black];
\draw (1.5,-.87) circle (2pt) [fill=black];
\draw (3,0) circle (2pt) [fill=black];
\draw (3,0)--(2.5,.87);
\draw (1.5,.87)--(2.5,.87);
\draw (1.5,.87)--(1,0);
\draw (1.5,-.87)--(1,0);
\draw (1.5,-.87)--(2.5,-.87);
\draw (3,0)--(2.5,-.87);

\draw (3.5,.87) circle (2pt) [fill=black];
\draw (4.5,.87) circle (2pt) [fill=black];
\draw (4.5,-.87) circle (2pt) [fill=black];
\draw (3.5,-.87) circle (2pt) [fill=black];
\draw (5,0) circle (2pt) [fill=black];
\draw (5,0)--(4.5,.87);
\draw (3.5,.87)--(4.5,.87);
\draw (3.5,.87)--(3,0);
\draw (3.5,-.87)--(3,0);
\draw (3.5,-.87)--(4.5,-.87);
\draw (5,0)--(4.5,-.87);

\draw (5.5,.87) circle (2pt) [fill=black];
\draw (6.5,.87) circle (2pt) [fill=black];
\draw (6.5,-.87) circle (2pt) [fill=black];
\draw (5.5,-.87) circle (2pt) [fill=black];
\draw (7,0) circle (2pt) [fill=black];
\draw (7,0)--(6.5,.87);
\draw (5.5,.87)--(6.5,.87);
\draw (5.5,.87)--(5,0);
\draw (5.5,-.87)--(5,0);
\draw (5.5,-.87)--(6.5,-.87);
\draw (7,0)--(6.5,-.87);

\draw (7.5,.87) circle (2pt) [fill=black];
\draw (8.5,.87) circle (2pt) [fill=black];
\draw (8.5,-.87) circle (2pt) [fill=black];
\draw (7.5,-.87) circle (2pt) [fill=black];
\draw (9,0) circle (2pt) [fill=black];
\draw (9,0)--(8.5,.87);
\draw (7.5,.87)--(8.5,.87);
\draw (7.5,.87)--(7,0);
\draw (7.5,-.87)--(7,0);
\draw (7.5,-.87)--(8.5,-.87);
\draw (9,0)--(8.5,-.87);

\draw (9.5,.87) circle (2pt) [fill=black];
\draw (10.5,-.87) circle (2pt) [fill=black];
\draw (9.5,-.87) circle (2pt) [fill=black];
\draw (9.5,.87)--(9,0);
\draw (9.5,-.87)--(9,0);
\draw (9.5,-.87)--(10.5,-.87);

\draw (0,0) node {$1$};
\draw (2,0) node {$2$};
\draw (4,0) node {$\cdots$};
\draw (8,0) node {$n$};
\end{tikzpicture}

\vspace{.5cm}

\begin{tikzpicture}[scale=.4]
\draw (-2.5,0) node {$P^2_n$};
\draw (-.5,.87) circle (2pt) [fill=black];
\draw (.5,.87) circle (2pt) [fill=black];
\draw (.5,-.87) circle (2pt) [fill=black];
\draw (-.5,-.87) circle (2pt) [fill=black];
\draw (1,0) circle (2pt) [fill=black];
\draw (-1,0) circle (2pt) [fill=black];
\draw (1,0)--(.5,.87);
\draw (-.5,.87)--(.5,.87);
\draw (-.5,.87)--(-1,0);
\draw (-.5,-.87)--(-1,0);
\draw (-.5,-.87)--(.5,-.87);
\draw (1,0)--(.5,-.87);

\draw (1.5,.87) circle (2pt) [fill=black];
\draw (2.5,.87) circle (2pt) [fill=black];
\draw (2.5,-.87) circle (2pt) [fill=black];
\draw (1.5,-.87) circle (2pt) [fill=black];
\draw (3,0) circle (2pt) [fill=black];
\draw (3,0)--(2.5,.87);
\draw (1.5,.87)--(2.5,.87);
\draw (1.5,.87)--(1,0);
\draw (1.5,-.87)--(1,0);
\draw (1.5,-.87)--(2.5,-.87);
\draw (3,0)--(2.5,-.87);

\draw (3.5,.87) circle (2pt) [fill=black];
\draw (4.5,.87) circle (2pt) [fill=black];
\draw (4.5,-.87) circle (2pt) [fill=black];
\draw (3.5,-.87) circle (2pt) [fill=black];
\draw (5,0) circle (2pt) [fill=black];
\draw (5,0)--(4.5,.87);
\draw (3.5,.87)--(4.5,.87);
\draw (3.5,.87)--(3,0);
\draw (3.5,-.87)--(3,0);
\draw (3.5,-.87)--(4.5,-.87);
\draw (5,0)--(4.5,-.87);

\draw (5.5,.87) circle (2pt) [fill=black];
\draw (6.5,.87) circle (2pt) [fill=black];
\draw (6.5,-.87) circle (2pt) [fill=black];
\draw (5.5,-.87) circle (2pt) [fill=black];
\draw (7,0) circle (2pt) [fill=black];
\draw (7,0)--(6.5,.87);
\draw (5.5,.87)--(6.5,.87);
\draw (5.5,.87)--(5,0);
\draw (5.5,-.87)--(5,0);
\draw (5.5,-.87)--(6.5,-.87);
\draw (7,0)--(6.5,-.87);

\draw (7.5,.87) circle (2pt) [fill=black];
\draw (8.5,.87) circle (2pt) [fill=black];
\draw (8.5,-.87) circle (2pt) [fill=black];
\draw (7.5,-.87) circle (2pt) [fill=black];
\draw (9,0) circle (2pt) [fill=black];
\draw (9,0)--(8.5,.87);
\draw (7.5,.87)--(8.5,.87);
\draw (7.5,.87)--(7,0);
\draw (7.5,-.87)--(7,0);
\draw (7.5,-.87)--(8.5,-.87);
\draw (9,0)--(8.5,-.87);

\draw (9.5,.87) circle (2pt) [fill=black];
\draw (9.5,.87)--(9,0);

\draw (0,0) node {$1$};
\draw (2,0) node {$2$};
\draw (4,0) node {$\cdots$};
\draw (8,0) node {$n$};
\end{tikzpicture}
\hspace{.5cm}
\begin{tikzpicture}[scale=.4]
\draw (-2.5,0) node {$P^3_n$};
\draw (-.5,.87) circle (2pt) [fill=black];
\draw (.5,.87) circle (2pt) [fill=black];
\draw (.5,-.87) circle (2pt) [fill=black];
\draw (-.5,-.87) circle (2pt) [fill=black];
\draw (1,0) circle (2pt) [fill=black];
\draw (-1,0) circle (2pt) [fill=black];
\draw (1,0)--(.5,.87);
\draw (-.5,.87)--(.5,.87);
\draw (-.5,.87)--(-1,0);
\draw (-.5,-.87)--(-1,0);
\draw (-.5,-.87)--(.5,-.87);
\draw (1,0)--(.5,-.87);

\draw (1.5,.87) circle (2pt) [fill=black];
\draw (2.5,.87) circle (2pt) [fill=black];
\draw (2.5,-.87) circle (2pt) [fill=black];
\draw (1.5,-.87) circle (2pt) [fill=black];
\draw (3,0) circle (2pt) [fill=black];
\draw (3,0)--(2.5,.87);
\draw (1.5,.87)--(2.5,.87);
\draw (1.5,.87)--(1,0);
\draw (1.5,-.87)--(1,0);
\draw (1.5,-.87)--(2.5,-.87);
\draw (3,0)--(2.5,-.87);

\draw (3.5,.87) circle (2pt) [fill=black];
\draw (4.5,.87) circle (2pt) [fill=black];
\draw (4.5,-.87) circle (2pt) [fill=black];
\draw (3.5,-.87) circle (2pt) [fill=black];
\draw (5,0) circle (2pt) [fill=black];
\draw (5,0)--(4.5,.87);
\draw (3.5,.87)--(4.5,.87);
\draw (3.5,.87)--(3,0);
\draw (3.5,-.87)--(3,0);
\draw (3.5,-.87)--(4.5,-.87);
\draw (5,0)--(4.5,-.87);

\draw (5.5,.87) circle (2pt) [fill=black];
\draw (6.5,.87) circle (2pt) [fill=black];
\draw (6.5,-.87) circle (2pt) [fill=black];
\draw (5.5,-.87) circle (2pt) [fill=black];
\draw (7,0) circle (2pt) [fill=black];
\draw (7,0)--(6.5,.87);
\draw (5.5,.87)--(6.5,.87);
\draw (5.5,.87)--(5,0);
\draw (5.5,-.87)--(5,0);
\draw (5.5,-.87)--(6.5,-.87);
\draw (7,0)--(6.5,-.87);

\draw (7.5,.87) circle (2pt) [fill=black];
\draw (8.5,.87) circle (2pt) [fill=black];
\draw (8.5,-.87) circle (2pt) [fill=black];
\draw (7.5,-.87) circle (2pt) [fill=black];
\draw (9,0) circle (2pt) [fill=black];
\draw (9,0)--(8.5,.87);
\draw (7.5,.87)--(8.5,.87);
\draw (7.5,.87)--(7,0);
\draw (7.5,-.87)--(7,0);
\draw (7.5,-.87)--(8.5,-.87);
\draw (9,0)--(8.5,-.87);

\draw (9.5,.87) circle (2pt) [fill=black];
\draw (10.5,.87) circle (2pt) [fill=black];
\draw (10.5,-.87) circle (2pt) [fill=black];
\draw (9.5,-.87) circle (2pt) [fill=black];
\draw (9.5,.87)--(10.5,.87);
\draw (9.5,.87)--(9,0);
\draw (9.5,-.87)--(9,0);
\draw (9.5,-.87)--(10.5,-.87);

\draw (0,0) node {$1$};
\draw (2,0) node {$2$};
\draw (4,0) node {$\cdots$};
\draw (8,0) node {$n$};
\end{tikzpicture}

\end{center}
\caption{Auxiliary graphs for $P_n$.} \label{auxpara}
\end{figure}
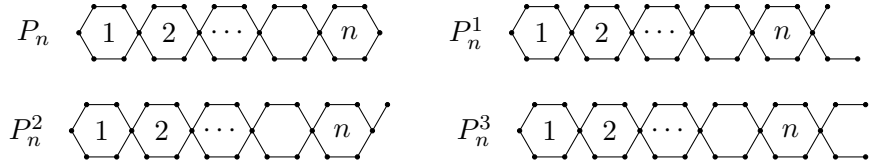

\begin{lemma} Let $p_n$ be the number of maximal matchings in $P_n$ and $p^i_n$ be the number of maximal matchings in the auxiliary graph $P^i_n$ in Figure \ref{auxpara}. Then\\

$(i)$ $p_n = 2p^1_{n-1} + p_{n-1}$, \\

$(ii)$ $p^1 _n = p^2_{n} + p^3_{n-1}$, \\

$(iii)$ $p^2 _n = p^3_{n-1} + 2p^1_{n-1}$, \\

$(iv)$ $p^3 _n = p_n + 2p^3_{n-1}$,\\

\noindent with the initial conditions $p_0=1$, $p^1_0=2$, $p^2_0=1$, and $p^3_0=3$.
\end{lemma}

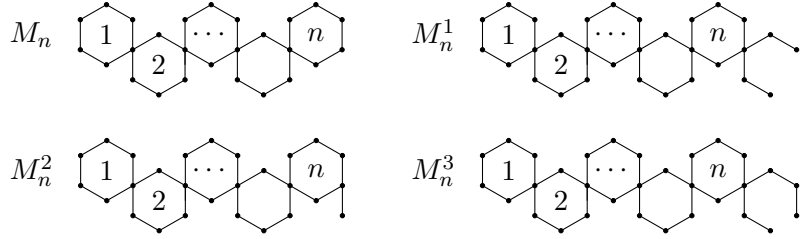
\begin{figure}[h]
\begin{center}
\begin{tikzpicture}[scale=.4]
\draw (-2.5,0) node {$M_n$};
\draw (.87,-.5) circle (2pt) [fill=black];
\draw (.87,.5) circle (2pt) [fill=black];
\draw (-.87,.5) circle (2pt) [fill=black];
\draw (-.87,-.5) circle (2pt) [fill=black];
\draw (0,1) circle (2pt) [fill=black];
\draw (0,-1) circle (2pt) [fill=black];
\draw (0,1)--(.87,.5);
\draw (.87,-.5)--(.87,.5);
\draw (.87,-.5)--(0,-1);
\draw (-.87,-.5)--(0,-1);
\draw (-.87,-.5)--(-.87,.5);
\draw (0,1)--(-.87,.5);

\draw (2.61,-.5) circle (2pt) [fill=black];
\draw (.87,-1.5) circle (2pt) [fill=black];
\draw (2.61,-1.5) circle (2pt) [fill=black];
\draw (1.74,0) circle (2pt) [fill=black];
\draw (1.74,-2) circle (2pt) [fill=black];
\draw (.87,-1.5)--(.87,-.5);
\draw (2.61,-1.5)--(1.74,-2);
\draw (2.61,-1.5)--(2.61,-.5);
\draw (1.74,0)--(2.61,-.5);
\draw (1.74,0)--(.87,-.5);
\draw (1.74,-2)--(.87,-1.5);

\draw (4.35,.5) circle (2pt) [fill=black];
\draw (4.35,-.5) circle (2pt) [fill=black];
\draw (3.48,1) circle (2pt) [fill=black];
\draw (3.48,-1) circle (2pt) [fill=black];
\draw (2.61,.5) circle (2pt) [fill=black];
\draw (2.61,.5)--(2.61,-1);
\draw (4.35,-.5)--(3.48,-1);
\draw (4.35,-.5)--(4.35,.5);
\draw (3.48,1)--(4.35,.5);
\draw (3.48,1)--(2.61,.5);
\draw (3.48,-1)--(2.61,-.5);

\draw (6.09,-.5) circle (2pt) [fill=black];
\draw (6.09,-1.5) circle (2pt) [fill=black];
\draw (5.22,0) circle (2pt) [fill=black];
\draw (5.22,-2) circle (2pt) [fill=black];
\draw (4.35,-1.5) circle (2pt) [fill=black];
\draw (4.35,-1.5)--(4.35,-.5);
\draw (6.09,-1.5)--(5.22,-2);
\draw (6.09,-1.5)--(6.09,-.5);
\draw (5.22,0)--(6.09,-.5);
\draw (5.22,0)--(4.35,-.5);
\draw (5.22,-2)--(4.35,-1.5);

\draw (7.83,.5) circle (2pt) [fill=black];
\draw (7.83,-.5) circle (2pt) [fill=black];
\draw (6.96,1) circle (2pt) [fill=black];
\draw (6.96,-1) circle (2pt) [fill=black];
\draw (6.09,.5) circle (2pt) [fill=black];
\draw (6.09,.5)--(6.09,-.5);
\draw (7.83,-.5)--(6.96,-1);
\draw (7.83,-.5)--(7.83,.5);
\draw (6.96,1)--(7.83,.5);
\draw (6.96,1)--(6.09,.5);
\draw (6.96,-1)--(6.09,-.5);

\draw (0,0) node {$1$};
\draw (1.74,-1) node {$2$};
\draw (3.48,0) node {$\cdots$};
\draw (6.96,0) node {$n$};
\end{tikzpicture}
\hspace{.5cm}
\begin{tikzpicture}[scale=.4]
\draw (-2.5,0) node {$M^1_n$};
\draw (.87,-.5) circle (2pt) [fill=black];
\draw (.87,.5) circle (2pt) [fill=black];
\draw (-.87,.5) circle (2pt) [fill=black];
\draw (-.87,-.5) circle (2pt) [fill=black];
\draw (0,1) circle (2pt) [fill=black];
\draw (0,-1) circle (2pt) [fill=black];
\draw (0,1)--(.87,.5);
\draw (.87,-.5)--(.87,.5);
\draw (.87,-.5)--(0,-1);
\draw (-.87,-.5)--(0,-1);
\draw (-.87,-.5)--(-.87,.5);
\draw (0,1)--(-.87,.5);

\draw (2.61,-.5) circle (2pt) [fill=black];
\draw (.87,-1.5) circle (2pt) [fill=black];
\draw (2.61,-1.5) circle (2pt) [fill=black];
\draw (1.74,0) circle (2pt) [fill=black];
\draw (1.74,-2) circle (2pt) [fill=black];
\draw (.87,-1.5)--(.87,-.5);
\draw (2.61,-1.5)--(1.74,-2);
\draw (2.61,-1.5)--(2.61,-.5);
\draw (1.74,0)--(2.61,-.5);
\draw (1.74,0)--(.87,-.5);
\draw (1.74,-2)--(.87,-1.5);

\draw (4.35,.5) circle (2pt) [fill=black];
\draw (4.35,-.5) circle (2pt) [fill=black];
\draw (3.48,1) circle (2pt) [fill=black];
\draw (3.48,-1) circle (2pt) [fill=black];
\draw (2.61,.5) circle (2pt) [fill=black];
\draw (2.61,.5)--(2.61,-1);
\draw (4.35,-.5)--(3.48,-1);
\draw (4.35,-.5)--(4.35,.5);
\draw (3.48,1)--(4.35,.5);
\draw (3.48,1)--(2.61,.5);
\draw (3.48,-1)--(2.61,-.5);

\draw (6.09,-.5) circle (2pt) [fill=black];
\draw (6.09,-1.5) circle (2pt) [fill=black];
\draw (5.22,0) circle (2pt) [fill=black];
\draw (5.22,-2) circle (2pt) [fill=black];
\draw (4.35,-1.5) circle (2pt) [fill=black];
\draw (4.35,-1.5)--(4.35,-.5);
\draw (6.09,-1.5)--(5.22,-2);
\draw (6.09,-1.5)--(6.09,-.5);
\draw (5.22,0)--(6.09,-.5);
\draw (5.22,0)--(4.35,-.5);
\draw (5.22,-2)--(4.35,-1.5);

\draw (7.83,.5) circle (2pt) [fill=black];
\draw (7.83,-.5) circle (2pt) [fill=black];
\draw (6.96,1) circle (2pt) [fill=black];
\draw (6.96,-1) circle (2pt) [fill=black];
\draw (6.09,.5) circle (2pt) [fill=black];
\draw (6.09,.5)--(6.09,-.5);
\draw (7.83,-.5)--(6.96,-1);
\draw (7.83,-.5)--(7.83,.5);
\draw (6.96,1)--(7.83,.5);
\draw (6.96,1)--(6.09,.5);
\draw (6.96,-1)--(6.09,-.5);

\draw (9.57,-.5) circle (2pt) [fill=black];
\draw (8.7,0) circle (2pt) [fill=black];
\draw (8.7,-2) circle (2pt) [fill=black];
\draw (7.83,-1.5) circle (2pt) [fill=black];
\draw (7.83,-1.5)--(7.83,-.5);
\draw (8.7,0)--(9.57,-.5);
\draw (8.7,0)--(7.83,-.5);
\draw (8.7,-2)--(7.83,-1.5);

\draw (0,0) node {$1$};
\draw (1.74,-1) node {$2$};
\draw (3.48,0) node {$\cdots$};
\draw (6.96,0) node {$n$};
\end{tikzpicture}

\vspace{.5cm}

\begin{tikzpicture}[scale=.4]
\draw (-2.5,0) node {$M^2_n$};
\draw (.87,-.5) circle (2pt) [fill=black];
\draw (.87,.5) circle (2pt) [fill=black];
\draw (-.87,.5) circle (2pt) [fill=black];
\draw (-.87,-.5) circle (2pt) [fill=black];
\draw (0,1) circle (2pt) [fill=black];
\draw (0,-1) circle (2pt) [fill=black];
\draw (0,1)--(.87,.5);
\draw (.87,-.5)--(.87,.5);
\draw (.87,-.5)--(0,-1);
\draw (-.87,-.5)--(0,-1);
\draw (-.87,-.5)--(-.87,.5);
\draw (0,1)--(-.87,.5);

\draw (2.61,-.5) circle (2pt) [fill=black];
\draw (.87,-1.5) circle (2pt) [fill=black];
\draw (2.61,-1.5) circle (2pt) [fill=black];
\draw (1.74,0) circle (2pt) [fill=black];
\draw (1.74,-2) circle (2pt) [fill=black];
\draw (.87,-1.5)--(.87,-.5);
\draw (2.61,-1.5)--(1.74,-2);
\draw (2.61,-1.5)--(2.61,-.5);
\draw (1.74,0)--(2.61,-.5);
\draw (1.74,0)--(.87,-.5);
\draw (1.74,-2)--(.87,-1.5);

\draw (4.35,.5) circle (2pt) [fill=black];
\draw (4.35,-.5) circle (2pt) [fill=black];
\draw (3.48,1) circle (2pt) [fill=black];
\draw (3.48,-1) circle (2pt) [fill=black];
\draw (2.61,.5) circle (2pt) [fill=black];
\draw (2.61,.5)--(2.61,-1);
\draw (4.35,-.5)--(3.48,-1);
\draw (4.35,-.5)--(4.35,.5);
\draw (3.48,1)--(4.35,.5);
\draw (3.48,1)--(2.61,.5);
\draw (3.48,-1)--(2.61,-.5);

\draw (6.09,-.5) circle (2pt) [fill=black];
\draw (6.09,-1.5) circle (2pt) [fill=black];
\draw (5.22,0) circle (2pt) [fill=black];
\draw (5.22,-2) circle (2pt) [fill=black];
\draw (4.35,-1.5) circle (2pt) [fill=black];
\draw (4.35,-1.5)--(4.35,-.5);
\draw (6.09,-1.5)--(5.22,-2);
\draw (6.09,-1.5)--(6.09,-.5);
\draw (5.22,0)--(6.09,-.5);
\draw (5.22,0)--(4.35,-.5);
\draw (5.22,-2)--(4.35,-1.5);

\draw (7.83,.5) circle (2pt) [fill=black];
\draw (7.83,-.5) circle (2pt) [fill=black];
\draw (6.96,1) circle (2pt) [fill=black];
\draw (6.96,-1) circle (2pt) [fill=black];
\draw (6.09,.5) circle (2pt) [fill=black];
\draw (6.09,.5)--(6.09,-.5);
\draw (7.83,-.5)--(6.96,-1);
\draw (7.83,-.5)--(7.83,.5);
\draw (6.96,1)--(7.83,.5);
\draw (6.96,1)--(6.09,.5);
\draw (6.96,-1)--(6.09,-.5);

\draw (7.83,-1.5) circle (2pt) [fill=black];
\draw (7.83,-1.5)--(7.83,-.5);

\draw (0,0) node {$1$};
\draw (1.74,-1) node {$2$};
\draw (3.48,0) node {$\cdots$};
\draw (6.96,0) node {$n$};
\end{tikzpicture}
\hspace{.5cm}
\begin{tikzpicture}[scale=.4]
\draw (-2.5,0) node {$M^3_n$};
\draw (.87,-.5) circle (2pt) [fill=black];
\draw (.87,.5) circle (2pt) [fill=black];
\draw (-.87,.5) circle (2pt) [fill=black];
\draw (-.87,-.5) circle (2pt) [fill=black];
\draw (0,1) circle (2pt) [fill=black];
\draw (0,-1) circle (2pt) [fill=black];
\draw (0,1)--(.87,.5);
\draw (.87,-.5)--(.87,.5);
\draw (.87,-.5)--(0,-1);
\draw (-.87,-.5)--(0,-1);
\draw (-.87,-.5)--(-.87,.5);
\draw (0,1)--(-.87,.5);

\draw (2.61,-.5) circle (2pt) [fill=black];
\draw (.87,-1.5) circle (2pt) [fill=black];
\draw (2.61,-1.5) circle (2pt) [fill=black];
\draw (1.74,0) circle (2pt) [fill=black];
\draw (1.74,-2) circle (2pt) [fill=black];
\draw (.87,-1.5)--(.87,-.5);
\draw (2.61,-1.5)--(1.74,-2);
\draw (2.61,-1.5)--(2.61,-.5);
\draw (1.74,0)--(2.61,-.5);
\draw (1.74,0)--(.87,-.5);
\draw (1.74,-2)--(.87,-1.5);

\draw (4.35,.5) circle (2pt) [fill=black];
\draw (4.35,-.5) circle (2pt) [fill=black];
\draw (3.48,1) circle (2pt) [fill=black];
\draw (3.48,-1) circle (2pt) [fill=black];
\draw (2.61,.5) circle (2pt) [fill=black];
\draw (2.61,.5)--(2.61,-1);
\draw (4.35,-.5)--(3.48,-1);
\draw (4.35,-.5)--(4.35,.5);
\draw (3.48,1)--(4.35,.5);
\draw (3.48,1)--(2.61,.5);
\draw (3.48,-1)--(2.61,-.5);

\draw (6.09,-.5) circle (2pt) [fill=black];
\draw (6.09,-1.5) circle (2pt) [fill=black];
\draw (5.22,0) circle (2pt) [fill=black];
\draw (5.22,-2) circle (2pt) [fill=black];
\draw (4.35,-1.5) circle (2pt) [fill=black];
\draw (4.35,-1.5)--(4.35,-.5);
\draw (6.09,-1.5)--(5.22,-2);
\draw (6.09,-1.5)--(6.09,-.5);
\draw (5.22,0)--(6.09,-.5);
\draw (5.22,0)--(4.35,-.5);
\draw (5.22,-2)--(4.35,-1.5);

\draw (7.83,.5) circle (2pt) [fill=black];
\draw (7.83,-.5) circle (2pt) [fill=black];
\draw (6.96,1) circle (2pt) [fill=black];
\draw (6.96,-1) circle (2pt) [fill=black];
\draw (6.09,.5) circle (2pt) [fill=black];
\draw (6.09,.5)--(6.09,-.5);
\draw (7.83,-.5)--(6.96,-1);
\draw (7.83,-.5)--(7.83,.5);
\draw (6.96,1)--(7.83,.5);
\draw (6.96,1)--(6.09,.5);
\draw (6.96,-1)--(6.09,-.5);

\draw (9.57,-.5) circle (2pt) [fill=black];
\draw (8.7,0) circle (2pt) [fill=black];
\draw (8.7,-2) circle (2pt) [fill=black];
\draw (7.83,-1.5) circle (2pt) [fill=black];
\draw (9.57,-1.5) circle (2pt) [fill=black];
\draw (9.57,-1.5)--(9.57,-.5);
\draw (7.83,-1.5)--(7.83,-.5);
\draw (8.7,0)--(9.57,-.5);
\draw (8.7,0)--(7.83,-.5);
\draw (8.7,-2)--(7.83,-1.5);

\draw (0,0) node {$1$};
\draw (1.74,-1) node {$2$};
\draw (3.48,0) node {$\cdots$};
\draw (6.96,0) node {$n$};
\end{tikzpicture}

\end{center}
\caption{Auxiliary graphs for $M_n$.} \label{auxmeta}
\end{figure}

\begin{lemma} Let $m_n$ be the number of maximal matchings in $M_n$ and $m^i_n$ be the number of maximal matchings in the auxiliary graph $M^i_n$ in Figure \ref{auxmeta}. Then\\

$(i)$ $m_n = 2m^1_{n-1} + m_{n-1}$, \\

$(ii)$ $m^1 _n = m^2_{n} + m^3_{n-1}$, \\

$(iii)$ $m^2 _n = m^3_{n-1} + m^1_{n-1} + m^2_{n-1} + m_{n-1}$, \\

$(iv)$ $m^3 _n = 2m^3_{n-1} + m^1_{n-1} + m^2_{n-1} + m_{n-1} + m^2_{n}$,\\

\noindent with the initial conditions $m_0=1,$ $m^1_0=2,$ $m^2_0=1,$ and $m^3_0=3$.
\end{lemma}

\begin{figure}[h]
\begin{center}
\begin{tikzpicture}[scale=.4]
\draw (-2.5,0) node {$O_n$};
\draw (-.7,.87) circle (2pt) [fill=black];
\draw (.7,.87) circle (2pt) [fill=black];
\draw (.7,-.87) circle (2pt) [fill=black];
\draw (-.7,-.87) circle (2pt) [fill=black];
\draw (1,0) circle (2pt) [fill=black];
\draw (-1,0) circle (2pt) [fill=black];
\draw (1,0)--(.7,.87);
\draw (-.7,.87)--(.7,.87);
\draw (-.7,.87)--(-1,0);
\draw (-.7,-.87)--(-1,0);
\draw (-.7,-.87)--(.7,-.87);
\draw (1,0)--(.7,-.87);

\draw (.7,-.87) circle (2pt) [fill=black];
\draw (2.1,-.87) circle (2pt) [fill=black];
\draw (2.1,-2.61) circle (2pt) [fill=black];
\draw (.7,-2.61) circle (2pt) [fill=black];
\draw (2.4,-1.74) circle (2pt) [fill=black];
\draw (.4,-1.74) circle (2pt) [fill=black];
\draw (2.4,-1.74)--(2.1,-.87);
\draw (.7,-.87)--(2.1,-.87);
\draw (.7,-.87)--(.4,-1.74);
\draw (.7,-2.61)--(.4,-1.74);
\draw (.7,-2.61)--(2.1,-2.61);
\draw (2.4,-1.74)--(2.1,-2.61);

\draw (2.1,.87) circle (2pt) [fill=black];
\draw (3.5,.87) circle (2pt) [fill=black];
\draw (3.5,-.87) circle (2pt) [fill=black];
\draw (2.1,-.87) circle (2pt) [fill=black];
\draw (3.8,0) circle (2pt) [fill=black];
\draw (1.8,0) circle (2pt) [fill=black];
\draw (3.8,0)--(3.5,.87);
\draw (2.1,.87)--(3.5,.87);
\draw (2.1,.87)--(1.8,0);
\draw (2.1,-.87)--(1.8,0);
\draw (2.1,-.87)--(3.5,-.87);
\draw (3.8,0)--(3.5,-.87);

\draw (3.5,-.87) circle (2pt) [fill=black];
\draw (4.9,-.87) circle (2pt) [fill=black];
\draw (4.9,-2.61) circle (2pt) [fill=black];
\draw (3.5,-2.61) circle (2pt) [fill=black];
\draw (5.2,-1.74) circle (2pt) [fill=black];
\draw (3.2,-1.74) circle (2pt) [fill=black];
\draw (5.2,-1.74)--(4.9,-.87);
\draw (3.5,-.87)--(4.9,-.87);
\draw (3.5,-.87)--(3.2,-1.74);
\draw (3.5,-2.61)--(3.2,-1.74);
\draw (3.5,-2.61)--(4.9,-2.61);
\draw (5.2,-1.74)--(4.9,-2.61);

\draw (4.9,.87) circle (2pt) [fill=black];
\draw (6.3,.87) circle (2pt) [fill=black];
\draw (6.3,-.87) circle (2pt) [fill=black];
\draw (4.9,-.87) circle (2pt) [fill=black];
\draw (6.6,0) circle (2pt) [fill=black];
\draw (4.6,0) circle (2pt) [fill=black];
\draw (6.6,0)--(6.3,.87);
\draw (4.9,.87)--(6.3,.87);
\draw (4.9,.87)--(4.6,0);
\draw (4.9,-.87)--(4.6,0);
\draw (4.9,-.87)--(6.3,-.87);
\draw (6.6,0)--(6.3,-.87);

\draw (0,0) node {$1$};
\draw (1.4,-1.74) node {$2$};
\draw (2.8,0) node {$\cdots$};
\draw (5.6,0) node {$n$};
\end{tikzpicture}
\hspace{.5cm}
\begin{tikzpicture}[scale=.4]
\draw (-2.5,0) node {$O^1_n$};
\draw (-.7,.87) circle (2pt) [fill=black];
\draw (.7,.87) circle (2pt) [fill=black];
\draw (.7,-.87) circle (2pt) [fill=black];
\draw (-.7,-.87) circle (2pt) [fill=black];
\draw (1,0) circle (2pt) [fill=black];
\draw (-1,0) circle (2pt) [fill=black];
\draw (1,0)--(.7,.87);
\draw (-.7,.87)--(.7,.87);
\draw (-.7,.87)--(-1,0);
\draw (-.7,-.87)--(-1,0);
\draw (-.7,-.87)--(.7,-.87);
\draw (1,0)--(.7,-.87);

\draw (.7,-.87) circle (2pt) [fill=black];
\draw (2.1,-.87) circle (2pt) [fill=black];
\draw (2.1,-2.61) circle (2pt) [fill=black];
\draw (.7,-2.61) circle (2pt) [fill=black];
\draw (2.4,-1.74) circle (2pt) [fill=black];
\draw (.4,-1.74) circle (2pt) [fill=black];
\draw (2.4,-1.74)--(2.1,-.87);
\draw (.7,-.87)--(2.1,-.87);
\draw (.7,-.87)--(.4,-1.74);
\draw (.7,-2.61)--(.4,-1.74);
\draw (.7,-2.61)--(2.1,-2.61);
\draw (2.4,-1.74)--(2.1,-2.61);

\draw (2.1,.87) circle (2pt) [fill=black];
\draw (3.5,.87) circle (2pt) [fill=black];
\draw (3.5,-.87) circle (2pt) [fill=black];
\draw (2.1,-.87) circle (2pt) [fill=black];
\draw (3.8,0) circle (2pt) [fill=black];
\draw (1.8,0) circle (2pt) [fill=black];
\draw (3.8,0)--(3.5,.87);
\draw (2.1,.87)--(3.5,.87);
\draw (2.1,.87)--(1.8,0);
\draw (2.1,-.87)--(1.8,0);
\draw (2.1,-.87)--(3.5,-.87);
\draw (3.8,0)--(3.5,-.87);

\draw (3.5,-.87) circle (2pt) [fill=black];
\draw (4.9,-.87) circle (2pt) [fill=black];
\draw (4.9,-2.61) circle (2pt) [fill=black];
\draw (3.5,-2.61) circle (2pt) [fill=black];
\draw (5.2,-1.74) circle (2pt) [fill=black];
\draw (3.2,-1.74) circle (2pt) [fill=black];
\draw (5.2,-1.74)--(4.9,-.87);
\draw (3.5,-.87)--(4.9,-.87);
\draw (3.5,-.87)--(3.2,-1.74);
\draw (3.5,-2.61)--(3.2,-1.74);
\draw (3.5,-2.61)--(4.9,-2.61);
\draw (5.2,-1.74)--(4.9,-2.61);

\draw (4.9,.87) circle (2pt) [fill=black];
\draw (6.3,.87) circle (2pt) [fill=black];
\draw (6.3,-.87) circle (2pt) [fill=black];
\draw (4.9,-.87) circle (2pt) [fill=black];
\draw (6.6,0) circle (2pt) [fill=black];
\draw (4.6,0) circle (2pt) [fill=black];
\draw (6.6,0)--(6.3,.87);
\draw (4.9,.87)--(6.3,.87);
\draw (4.9,.87)--(4.6,0);
\draw (4.9,-.87)--(4.6,0);
\draw (4.9,-.87)--(6.3,-.87);
\draw (6.6,0)--(6.3,-.87);

\draw (6.3,-.87) circle (2pt) [fill=black];
\draw (7.7,-.87) circle (2pt) [fill=black];
\draw (6.3,-2.61) circle (2pt) [fill=black];
\draw (6,-1.74) circle (2pt) [fill=black];
\draw (6.3,-.87)--(7.7,-.87);
\draw (6.3,-.87)--(6,-1.74);
\draw (6.3,-2.61)--(6,-1.74);

\draw (0,0) node {$1$};
\draw (1.4,-1.74) node {$2$};
\draw (2.8,0) node {$\cdots$};
\draw (5.6,0) node {$n$};
\end{tikzpicture}

\vspace{.5cm}

\begin{tikzpicture}[scale=.4]
\draw (-2.5,0) node {$O^2_n$};
\draw (-.7,.87) circle (2pt) [fill=black];
\draw (.7,.87) circle (2pt) [fill=black];
\draw (.7,-.87) circle (2pt) [fill=black];
\draw (-.7,-.87) circle (2pt) [fill=black];
\draw (1,0) circle (2pt) [fill=black];
\draw (-1,0) circle (2pt) [fill=black];
\draw (1,0)--(.7,.87);
\draw (-.7,.87)--(.7,.87);
\draw (-.7,.87)--(-1,0);
\draw (-.7,-.87)--(-1,0);
\draw (-.7,-.87)--(.7,-.87);
\draw (1,0)--(.7,-.87);

\draw (.7,-.87) circle (2pt) [fill=black];
\draw (2.1,-.87) circle (2pt) [fill=black];
\draw (2.1,-2.61) circle (2pt) [fill=black];
\draw (.7,-2.61) circle (2pt) [fill=black];
\draw (2.4,-1.74) circle (2pt) [fill=black];
\draw (.4,-1.74) circle (2pt) [fill=black];
\draw (2.4,-1.74)--(2.1,-.87);
\draw (.7,-.87)--(2.1,-.87);
\draw (.7,-.87)--(.4,-1.74);
\draw (.7,-2.61)--(.4,-1.74);
\draw (.7,-2.61)--(2.1,-2.61);
\draw (2.4,-1.74)--(2.1,-2.61);

\draw (2.1,.87) circle (2pt) [fill=black];
\draw (3.5,.87) circle (2pt) [fill=black];
\draw (3.5,-.87) circle (2pt) [fill=black];
\draw (2.1,-.87) circle (2pt) [fill=black];
\draw (3.8,0) circle (2pt) [fill=black];
\draw (1.8,0) circle (2pt) [fill=black];
\draw (3.8,0)--(3.5,.87);
\draw (2.1,.87)--(3.5,.87);
\draw (2.1,.87)--(1.8,0);
\draw (2.1,-.87)--(1.8,0);
\draw (2.1,-.87)--(3.5,-.87);
\draw (3.8,0)--(3.5,-.87);

\draw (3.5,-.87) circle (2pt) [fill=black];
\draw (4.9,-.87) circle (2pt) [fill=black];
\draw (4.9,-2.61) circle (2pt) [fill=black];
\draw (3.5,-2.61) circle (2pt) [fill=black];
\draw (5.2,-1.74) circle (2pt) [fill=black];
\draw (3.2,-1.74) circle (2pt) [fill=black];
\draw (5.2,-1.74)--(4.9,-.87);
\draw (3.5,-.87)--(4.9,-.87);
\draw (3.5,-.87)--(3.2,-1.74);
\draw (3.5,-2.61)--(3.2,-1.74);
\draw (3.5,-2.61)--(4.9,-2.61);
\draw (5.2,-1.74)--(4.9,-2.61);

\draw (4.9,.87) circle (2pt) [fill=black];
\draw (6.3,.87) circle (2pt) [fill=black];
\draw (6.3,-.87) circle (2pt) [fill=black];
\draw (4.9,-.87) circle (2pt) [fill=black];
\draw (6.6,0) circle (2pt) [fill=black];
\draw (4.6,0) circle (2pt) [fill=black];
\draw (6.6,0)--(6.3,.87);
\draw (4.9,.87)--(6.3,.87);
\draw (4.9,.87)--(4.6,0);
\draw (4.9,-.87)--(4.6,0);
\draw (4.9,-.87)--(6.3,-.87);
\draw (6.6,0)--(6.3,-.87);

\draw (6.3,-.87) circle (2pt) [fill=black];
\draw (7.7,-.87) circle (2pt) [fill=black];
\draw (6.3,-.87)--(7.7,-.87);

\draw (0,0) node {$1$};
\draw (1.4,-1.74) node {$2$};
\draw (2.8,0) node {$\cdots$};
\draw (5.6,0) node {$n$};
\end{tikzpicture}
\hspace{.5cm}
\begin{tikzpicture}[scale=.4]
\draw (-2.5,0) node {$O^3_n$};
\draw (-.7,.87) circle (2pt) [fill=black];
\draw (.7,.87) circle (2pt) [fill=black];
\draw (.7,-.87) circle (2pt) [fill=black];
\draw (-.7,-.87) circle (2pt) [fill=black];
\draw (1,0) circle (2pt) [fill=black];
\draw (-1,0) circle (2pt) [fill=black];
\draw (1,0)--(.7,.87);
\draw (-.7,.87)--(.7,.87);
\draw (-.7,.87)--(-1,0);
\draw (-.7,-.87)--(-1,0);
\draw (-.7,-.87)--(.7,-.87);
\draw (1,0)--(.7,-.87);

\draw (.7,-.87) circle (2pt) [fill=black];
\draw (2.1,-.87) circle (2pt) [fill=black];
\draw (2.1,-2.61) circle (2pt) [fill=black];
\draw (.7,-2.61) circle (2pt) [fill=black];
\draw (2.4,-1.74) circle (2pt) [fill=black];
\draw (.4,-1.74) circle (2pt) [fill=black];
\draw (2.4,-1.74)--(2.1,-.87);
\draw (.7,-.87)--(2.1,-.87);
\draw (.7,-.87)--(.4,-1.74);
\draw (.7,-2.61)--(.4,-1.74);
\draw (.7,-2.61)--(2.1,-2.61);
\draw (2.4,-1.74)--(2.1,-2.61);

\draw (2.1,.87) circle (2pt) [fill=black];
\draw (3.5,.87) circle (2pt) [fill=black];
\draw (3.5,-.87) circle (2pt) [fill=black];
\draw (2.1,-.87) circle (2pt) [fill=black];
\draw (3.8,0) circle (2pt) [fill=black];
\draw (1.8,0) circle (2pt) [fill=black];
\draw (3.8,0)--(3.5,.87);
\draw (2.1,.87)--(3.5,.87);
\draw (2.1,.87)--(1.8,0);
\draw (2.1,-.87)--(1.8,0);
\draw (2.1,-.87)--(3.5,-.87);
\draw (3.8,0)--(3.5,-.87);

\draw (3.5,-.87) circle (2pt) [fill=black];
\draw (4.9,-.87) circle (2pt) [fill=black];
\draw (4.9,-2.61) circle (2pt) [fill=black];
\draw (3.5,-2.61) circle (2pt) [fill=black];
\draw (5.2,-1.74) circle (2pt) [fill=black];
\draw (3.2,-1.74) circle (2pt) [fill=black];
\draw (5.2,-1.74)--(4.9,-.87);
\draw (3.5,-.87)--(4.9,-.87);
\draw (3.5,-.87)--(3.2,-1.74);
\draw (3.5,-2.61)--(3.2,-1.74);
\draw (3.5,-2.61)--(4.9,-2.61);
\draw (5.2,-1.74)--(4.9,-2.61);

\draw (4.9,.87) circle (2pt) [fill=black];
\draw (6.3,.87) circle (2pt) [fill=black];
\draw (6.3,-.87) circle (2pt) [fill=black];
\draw (4.9,-.87) circle (2pt) [fill=black];
\draw (6.6,0) circle (2pt) [fill=black];
\draw (4.6,0) circle (2pt) [fill=black];
\draw (6.6,0)--(6.3,.87);
\draw (4.9,.87)--(6.3,.87);
\draw (4.9,.87)--(4.6,0);
\draw (4.9,-.87)--(4.6,0);
\draw (4.9,-.87)--(6.3,-.87);
\draw (6.6,0)--(6.3,-.87);

\draw (6.3,-.87) circle (2pt) [fill=black];
\draw (7.7,-2.61) circle (2pt) [fill=black];
\draw (6.3,-2.61) circle (2pt) [fill=black];
\draw (8,-1.74) circle (2pt) [fill=black];
\draw (6,-1.74) circle (2pt) [fill=black];

\draw (6.3,-.87)--(6,-1.74);
\draw (6.3,-2.61)--(6,-1.74);
\draw (6.3,-2.61)--(7.7,-2.61);
\draw (8,-1.74)--(7.7,-2.61);

\draw (0,0) node {$1$};
\draw (1.4,-1.74) node {$2$};
\draw (2.8,0) node {$\cdots$};
\draw (5.6,0) node {$n$};
\end{tikzpicture}

\end{center}
\caption{Auxiliary graphs for $O_n$.} \label{auxortho}
\end{figure}
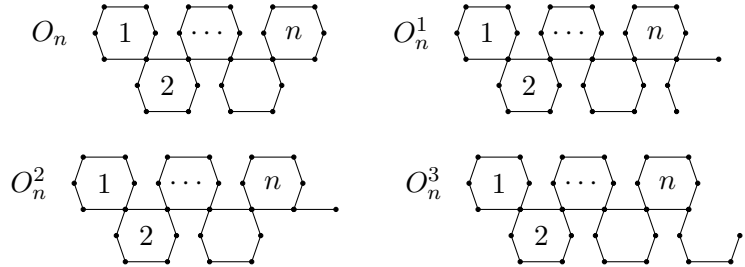

\begin{lemma} Let $o_n$ be the number of maximal matchings in $O_n$ and $o^i_n$ be the number of maximal matchings in the auxiliary graph $O^i_n$ in Figure \ref{auxortho}. Then\\

$(i)$ $o_n = 2o^1_{n-1} + o_{n-1}$, \\

$(ii)$ $o^1 _n = o^2_{n} + o^3_{n-1}$, \\

$(iii)$ $o^2 _n = o^3_{n-1} + o^2_{n-1} + o_{n-1} + 2o^3_{n-2}$, \\

$(iv)$ $o^3 _n = o_n + o^3_{n-1} + o^2_n$,\\

\noindent with the initial conditions $o_0=1$, $o^1_0=2$, $o^2_0=1$, $o^2_1=7$, and $o^3_0=3$.
\end{lemma}

\begin{theorem} \label{cactigen}
Let $P(x)$, $M(x)$, and $O(x)$ be the ordinary generating functions for the sequences $p_n$, $m_n$, and $o_n$, respectively. Then\\

$(i)$ 
$$
P(x) = \frac{1 + 4x^2}{1 - 5x +4x^2-4x^3},
$$

$(ii)$
$$
M(x) = \frac{1 - x - 2x^2}{1-6x+3x^2-2x^3},
$$

$(iii)$
$$
O(x) = \frac{1 + x + x^2}{1 - 4x - 4x^2 - x^3}.
$$
\end{theorem}

\noindent Since $P(x)$, $M(x)$, and $O(x)$ are rational functions, we can conclude that the numbers $p_n$, $m_n$, and $o_n$ each satisfy a third order linear recurrence with constant coefficients. The initial conditions can be verified by direct computations.

\begin{corollary}\ \\

$(i)$ $p_n = 5p_{n-1}-4p_{n-2}+4p_{n-3}$\\

\noindent with initial conditions $p_0 = 1$, $p_1 = 5$, $p_2 = 25$,\\

$(ii)$ $m_n = 6m_{n-1} - 3m_{n-2} + 2m_{n-3}$\\

\noindent with initial conditions $m_0 = 1$, $m_1 = 5$, $m_2 = 25$,\\

$(iii)$ $o_n = 4o_{n-1} + 4o_{n-2} + o_{n-3}$\\

\noindent with initial conditions $o_0 = 1$, $o_1 = 5$, $o_2 = 25$.
\end{corollary}

\noindent None of the obtained sequences appear in {\em The On-Line Encyclopedia of Integer Sequences} \cite{oeis}.

Now we can apply a version of Darboux's theorem to deduce the asymptotic behavior of the sequences $p_n$, $m_n$, and $o_n$. We refer the reader to any of standard books on generating functions, such as \cite{bgw, wilf} for more information on these techniques.

\begin{theorem}[Darboux]
Let $f(x) = \sum _{n=0}^\infty a_nx^n$ denote the ordinary generating function of a sequence $a_n$. If $f(x)$ can be written as
$$
f(x) = \left( 1 - \frac{x}{w} \right)^\alpha g(x),
$$
where $w$ is the smallest modulus singularity of $f$ and $g$ is analytic at $w$, then
$$
a_n \sim \frac{g(w)}{\Gamma (-\alpha)} w^{-n}n^{-\alpha - 1}.
$$
Here $\Gamma (x)$ denotes the gamma function.
\end{theorem}

\begin{corollary}\ \\

$(i)$ $p_n \sim 1.37804 \cdot 4.28428^n$,\\

$(ii)$ $m_n \sim 0.81408 \cdot 5.52233^n$,\\

$(iii)$ $o_n \sim 1.05177 \cdot 4.86454^n$.

\end{corollary}

The characteristic equations of the three recurrences can be solved exactly,
but the resulting formulas tend to be too cumbersome to be of any use. 
The equation for meta-chains, however, allows a compact formula for the
smallest (and the only) positive root: it is equal to $\frac{1}{2}(1 +
\sqrt[3]3 - \sqrt [3]9)$.

The obtained asymptotics suggest that meta-chains could be the richest and
para-chains the poorest in maximal matchings among all polyspiro chains of the
same length. In the next subsection we prove that this is, indeed, the case.

\subsection{Extremal structures} \label{secextremecacti}

\begin{theorem} \label{extremecacti}
Let $G_n$ be a hexagonal cactus of length $n$. Then
$$
\Psi (P_n) \le  \Psi (G_n) \le \Psi (M_n).
$$
\end{theorem}

Let $G_m$ be an arbitrary hexagonal cactus of length $m$. Observe that we can always draw $G_m$ as in Figure \ref{cactus}, where $h_m$ is a terminal hexagon and the hexagon adjacent to the left of $h_{m-1}$ may attach at any of the vertices $b, a, k, j,$ or $i$. Let us assume the hexagons of $G_m$ are labeled $h_1, \ldots, h_m$ according to their ordering in Figure \ref{cactus} where ($h_1$ is the other terminal hexagon).

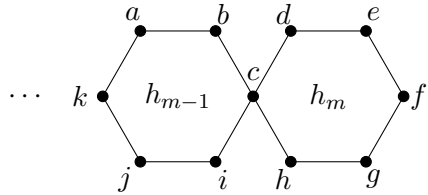
\begin{figure}[h]
\begin{center}
\begin{tikzpicture}[scale=1]

\draw (-.5,.87) circle (2pt) [fill=black];
\draw (.5,.87) circle (2pt) [fill=black];
\draw (.5,-.87) circle (2pt) [fill=black];
\draw (-.5,-.87) circle (2pt) [fill=black];
\draw (1,0) circle (2pt) [fill=black];
\draw (-1,0) circle (2pt) [fill=black];
\draw (1,0)--(.5,.87);
\draw (-.5,.87)--(.5,.87);
\draw (-.5,.87)--(-1,0);
\draw (-.5,-.87)--(-1,0);
\draw (-.5,-.87)--(.5,-.87);
\draw (1,0)--(.5,-.87);

\draw (1.5,.87) circle (2pt) [fill=black];
\draw (2.5,.87) circle (2pt) [fill=black];
\draw (2.5,-.87) circle (2pt) [fill=black];
\draw (1.5,-.87) circle (2pt) [fill=black];
\draw (3,0) circle (2pt) [fill=black];
\draw (3,0)--(2.5,.87);
\draw (1.5,.87)--(2.5,.87);
\draw (1.5,.87)--(1,0);
\draw (1.5,-.87)--(1,0);
\draw (1.5,-.87)--(2.5,-.87);
\draw (3,0)--(2.5,-.87);

\draw (0,0) node {$h_{m-1}$};
\draw (2,0) node {$h_{m}$};
\draw (-2,0) node {$\cdots$};

\draw (-.6,1.1) node {$a$};
\draw (.6,1.1) node {$b$};
\draw (1,.3) node {$c$};
\draw (1.4,1.1) node {$d$};
\draw (2.6,1.1) node {$e$};
\draw (3.2,0) node {$f$};
\draw (-.7,-1.1) node {$j$};
\draw (.6,-1.1) node {$i$};
\draw (1.4,-1.1) node {$h$};
\draw (2.6,-1.1) node {$g$};
\draw (-1.3,0) node {$k$};

\end{tikzpicture}
\end{center}
\caption{A terminal hexagon, $h_m$, and its adjacent hexagon, $h_{m-1}$, in the hexagonal chain cactus $G_m$.} \label{cactus}
\end{figure}

In what follows, for $1 \le \ell, p \le m$ let $H_\ell$ be the subgraph of $G_m$ induced by the vertices of the hexagons $h_1, \ldots, h_\ell$ and let $H_{\ell,p}$ denote the subgraph of $G_m$ induced by the vertices of the two hexagons $h_\ell$ and $h_p$. We will need the following lemmas. The proof of the first lemma is immediate.

\begin{lemma} \label{subgraph}
If $H$ is a subgraph of the graph $G$, then $\Psi(H) \le \Psi (G)$.
\end{lemma}

\begin{lemma} \label{cactiedges}
Any maximal matching in $G_m$ must contain exactly one of the edges $cb, cd, ch,$ or $ci$, or the maximal matching must contain all the edges $ab, de, ji,$ and $hg$. 
\end{lemma}

\begin{proof}
Take a maximal matching $M$ in $G_m$. For sake of contradiction, suppose that $M$ does not contain any of the edges $cb$, $cd$, $ch$, or $ci$ and that $M$ does not contain all of the edges $ab$, $de$, $ji$, and $hg$. Then at least one of the edges $ab$, $de$, $ji$, and $hg$ is missing, say $ab$. Since $ab$ is not in $M$, then we can add the edge $bc$ to $M$, which is a contradiction to the fact that $M$ is a maximal matching. The lemma follows.
\end{proof}

\begin{lemma} \label{subgraph2}
For the subgraph $H_{m-1}$ of $G_m$, at least one of the following holds:

$(i)$ $2 \cdot \Psi (H_{m-1} - \{b,c\}) \ge \Psi (H_{m-1} - c)$

$(ii)$ $2 \cdot \Psi (H_{m-1} - \{c,i\}) \ge \Psi (H_{m-1} - c)$
\end{lemma}

\begin{proof}
The proof depends on where the hexagon $h_{m-2}$ attaches to $h_{m-1}$. By symmetry, suppose that $h_{m-2}$ attaches at either $i, j,$ or $k$ (the case $a,b,k$ is similar). Consider a maximal matching of $H_{m-1}-c$. If such a matching contains the edge $ab$, then the remaining edges give a maximal matching of $H_{m-1} - \{a,b,c\}$. If a maximal matching does not contain the edge $ab$, then the matching must also be maximal in the graph $H_{m-1} - \{b,c\}$. Thus by Lemma \ref{subgraph} we have
\begin{align*}
\Psi (H_{m-1} - \{c\}) &= \Psi (H_{m-1} - \{a,b,c\}) + \Psi (H_{m-1} - \{b,c\}) \\
&\le 2 \cdot \Psi (H_{m-1} - \{b,c\}).
\end{align*}
\end{proof} 

\begin{proof}[Proof (of Theorem \ref{extremecacti})]
Take a hexagonal cactus $C$ of length $n-1$. Let us set $m=n-1$ and suppose that $C$ is drawn as in Figure \ref{cactus} with vertices labeled as such, so that we may refer to this picture to aid this proof. We consider three cases of extending $C$ by an $n$th hexagon $h_n$.

\begin{figure}[h]
\begin{center}
\begin{tikzpicture}[scale=1]

\draw (-.5,.87) circle (2pt) [fill=black];
\draw (.5,.87) circle (2pt) [fill=black];
\draw (.5,-.87) circle (2pt) [fill=black];
\draw (-.5,-.87) circle (2pt) [fill=black];
\draw (1,0) circle (2pt) [fill=black];
\draw (-1,0) circle (2pt) [fill=black];
\draw (1,0)--(.5,.87);
\draw (-.5,.87)--(.5,.87);
\draw (-.5,.87)--(-1,0);
\draw (-.5,-.87)--(-1,0);
\draw (-.5,-.87)--(.5,-.87);
\draw (1,0)--(.5,-.87);

\draw (1.5,.87) circle (2pt) [fill=black];
\draw (2.5,.87) circle (2pt) [fill=black];
\draw (2.5,-.87) circle (2pt) [fill=black];
\draw (1.5,-.87) circle (2pt) [fill=black];
\draw (3,0) circle (2pt) [fill=black];
\draw (3,0)--(2.5,.87);
\draw (1.5,.87)--(2.5,.87);
\draw (1.5,.87)--(1,0);
\draw (1.5,-.87)--(1,0);
\draw (1.5,-.87)--(2.5,-.87);
\draw (3,0)--(2.5,-.87);

\draw (0,0) node {$h_{n-2}$};
\draw (2,0) node {$h_{n-1}$};
\draw (4,0) node {$h_{n}$};
\draw (-2,0) node {$\cdots$};

\draw (-.6,1.1) node {$a$};
\draw (.6,1.1) node {$b$};
\draw (1,.3) node {$c$};
\draw (1.4,1.1) node {$d$};
\draw (2.6,1.1) node {$e$};
\draw (3,.3) node {$f$};
\draw (-.7,-1.1) node {$j$};
\draw (.6,-1.1) node {$i$};
\draw (1.4,-1.1) node {$h$};
\draw (2.6,-1.1) node {$g$};
\draw (-1.3,0) node {$k$};

\draw (3.5,.87) circle (2pt) [fill=black];
\draw (4.5,.87) circle (2pt) [fill=black];
\draw (4.5,-.87) circle (2pt) [fill=black];
\draw (3.5,-.87) circle (2pt) [fill=black];
\draw (5,0) circle (2pt) [fill=black];
\draw (5,0)--(4.5,.87);
\draw (3.5,.87)--(4.5,.87);
\draw (3.5,.87)--(3,0);
\draw (3.5,-.87)--(3,0);
\draw (3.5,-.87)--(4.5,-.87);
\draw (5,0)--(4.5,-.87);

\end{tikzpicture}
\end{center}
\caption{The hexagonal cactus CP.} \label{CP}
\end{figure}
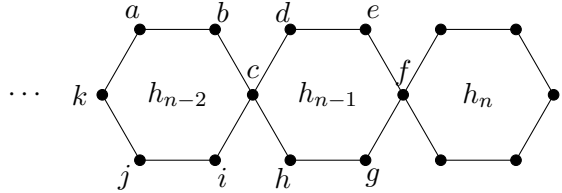

{\bf Case 1.} The hexagon $h_n$ attaches in the para position to the vertex $f$ and let us denote the resulting graph by $CP$, see Figure \ref{CP}. To compute $\Psi (CP)$ we make use of Lemma \ref{cactiedges}. Consider maximal matchings in $CP$ containing the edge $bc$. The remaining edges of the matching must be a maximal matching of $H_{n-2}-\{b,c\}$ and a maximal matching of $H_{n-1,n}-c$. By direct counting, we find that $\Psi (H_{n-1,n}-c) = 11$ and hence, the number of maximal matchings containing the edge $bc$ is $11\cdot \Psi(H_{n-2}-\{b,c\})$. We count the maximal matchings containing the edges $ci, cd,$ or $ch$ as well as the maximal matchings containing all the edges $ab, de, ji,$ and $hg$ similarly, to obtain

\begin{align*}
\Psi (CP) = & 11(\Psi(H_{n-2}-\{b,c\}) + \Psi(H_{n-2}-\{c,i\})) + 20\cdot \Psi (H_{n-2} - c) \\
&+ 5\cdot \Psi (H_{n-2}-\{a,b,c,i,j\}).
\end{align*}

\begin{figure}[h]
\begin{center}
\begin{tikzpicture}[scale=1]

\draw (-.5,.87) circle (2pt) [fill=black];
\draw (.5,.87) circle (2pt) [fill=black];
\draw (.5,-.87) circle (2pt) [fill=black];
\draw (-.5,-.87) circle (2pt) [fill=black];
\draw (1,0) circle (2pt) [fill=black];
\draw (-1,0) circle (2pt) [fill=black];
\draw (1,0)--(.5,.87);
\draw (-.5,.87)--(.5,.87);
\draw (-.5,.87)--(-1,0);
\draw (-.5,-.87)--(-1,0);
\draw (-.5,-.87)--(.5,-.87);
\draw (1,0)--(.5,-.87);

\draw (1.5,.87) circle (2pt) [fill=black];
\draw (2.5,.87) circle (2pt) [fill=black];
\draw (2.5,-.87) circle (2pt) [fill=black];
\draw (1.5,-.87) circle (2pt) [fill=black];
\draw (3,0) circle (2pt) [fill=black];
\draw (3,0)--(2.5,.87);
\draw (1.5,.87)--(2.5,.87);
\draw (1.5,.87)--(1,0);
\draw (1.5,-.87)--(1,0);
\draw (1.5,-.87)--(2.5,-.87);
\draw (3,0)--(2.5,-.87);

\draw (0,0) node {$h_{n-2}$};
\draw (2,0) node {$h_{n-1}$};
\draw (3,1.74) node {$h_{n}$};
\draw (-2,0) node {$\cdots$};

\draw (-.6,1.1) node {$a$};
\draw (.6,1.1) node {$b$};
\draw (1,.3) node {$c$};
\draw (1.4,1.1) node {$d$};
\draw (2.6,1.1) node {$e$};
\draw (3,.3) node {$f$};
\draw (-.7,-1.1) node {$j$};
\draw (.6,-1.1) node {$i$};
\draw (1.4,-1.1) node {$h$};
\draw (2.6,-1.1) node {$g$};
\draw (-1.3,0) node {$k$};

\draw (2.5,2.61) circle (2pt) [fill=black];
\draw (3.5,2.61) circle (2pt) [fill=black];
\draw (3.5,.87) circle (2pt) [fill=black];
\draw (2.5,.87) circle (2pt) [fill=black];
\draw (4,1.74) circle (2pt) [fill=black];
\draw (2,1.74) circle (2pt) [fill=black];
\draw (4,1.74)--(3.5,2.61);
\draw (2.5,2.61)--(3.5,2.61);
\draw (2.5,2.61)--(2,1.74);
\draw (2.5,.87)--(2,1.74);
\draw (2.5,.87)--(3.5,.87);
\draw (4,1.74)--(3.5,.87);

\end{tikzpicture}
\end{center}
\caption{The hexagonal cactus CM.} \label{CM}
\end{figure}
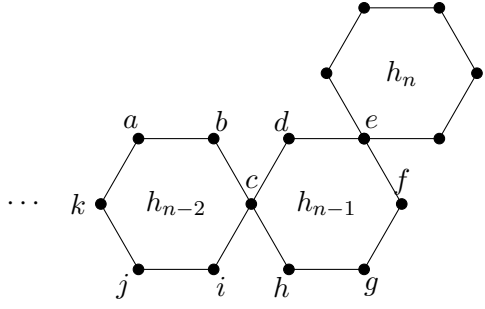

{\bf Case 2.} The hexagon $h_n$ attaches in the meta position to the vertex $e$ and let us denote the resulting graph by $CM$, see Figure \ref{CM}. Counting similarly to Case 1 above we obtain
\begin{align*}
\Psi (CM) = & 17(\Psi(H_{n-2}-\{b,c\}) + \Psi(H_{n-2}-\{c,i\})) + 22\cdot \Psi (H_{n-2} - c) \\
&+ 3\cdot \Psi (H_{n-2}-\{a,b,c,i,j\}).
\end{align*}

\begin{figure}[h]
\begin{center}
\begin{tikzpicture}[scale=1]

\draw (-.5,.87) circle (2pt) [fill=black];
\draw (.5,.87) circle (2pt) [fill=black];
\draw (.5,-.87) circle (2pt) [fill=black];
\draw (-.5,-.87) circle (2pt) [fill=black];
\draw (1,0) circle (2pt) [fill=black];
\draw (-1,0) circle (2pt) [fill=black];
\draw (1,0)--(.5,.87);
\draw (-.5,.87)--(.5,.87);
\draw (-.5,.87)--(-1,0);
\draw (-.5,-.87)--(-1,0);
\draw (-.5,-.87)--(.5,-.87);
\draw (1,0)--(.5,-.87);

\draw (1.5,.87) circle (2pt) [fill=black];
\draw (2.5,.87) circle (2pt) [fill=black];
\draw (2.5,-.87) circle (2pt) [fill=black];
\draw (1.5,-.87) circle (2pt) [fill=black];
\draw (3,0) circle (2pt) [fill=black];
\draw (3,0)--(2.5,.87);
\draw (1.5,.87)--(2.5,.87);
\draw (1.5,.87)--(1,0);
\draw (1.5,-.87)--(1,0);
\draw (1.5,-.87)--(2.5,-.87);
\draw (3,0)--(2.5,-.87);

\draw (0,0) node {$h_{n-2}$};
\draw (2,0) node {$h_{n-1}$};
\draw (1.5,1.87) node {$h_{n}$};
\draw (-2,0) node {$\cdots$};

\draw (-.6,1.1) node {$a$};
\draw (.6,1.1) node {$b$};
\draw (1,.3) node {$c$};
\draw (1.5,1.2) node {$d$};
\draw (2.6,1.1) node {$e$};
\draw (3,.3) node {$f$};
\draw (-.7,-1.1) node {$j$};
\draw (.6,-1.1) node {$i$};
\draw (1.4,-1.1) node {$h$};
\draw (2.6,-1.1) node {$g$};
\draw (-1.3,0) node {$k$};

\draw (1.5,.87) circle (2pt) [fill=black];
\draw (.63,1.37) circle (2pt) [fill=black];
\draw (2.37,1.37) circle (2pt) [fill=black];
\draw (.63,2.37) circle (2pt) [fill=black];
\draw (2.37,2.37) circle (2pt) [fill=black];
\draw (1.5,2.87) circle (2pt) [fill=black];
\draw (1.5,.87)--(.63,1.37);
\draw (1.5,.87)--(2.37,1.37);
\draw (.63,1.37)--(.63,2.37);
\draw (1.5,2.87)--(.63,2.37);
\draw (1.5,2.87)--(2.37,2.37);
\draw (2.37,2.37)--(2.37,1.37);

\end{tikzpicture}
\end{center}
\caption{The hexagonal cactus CO.} \label{CO}
\end{figure}
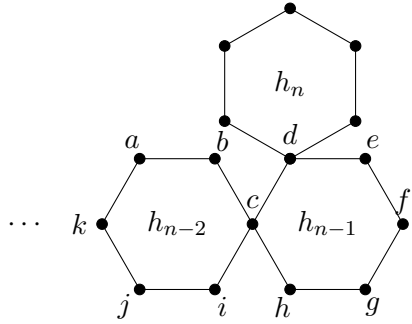

{\bf Case 3.} The hexagon $h_n$ attaches in the ortho position to the vertex $d$ and let us denote the resulting graph by $CO$, see Figure \ref{CO}. Counting as in Cases 1 and 2,
\begin{align*}
\Psi (CO) = & 15(\Psi(H_{n-2}-\{b,c\}) + \Psi(H_{n-2}-\{c,i\})) + 18\cdot \Psi (H_{n-2} - c)\\
&+ 3\cdot \Psi (H_{n-2}-\{a,b,c,i,j\}).
\end{align*}

Now $\Psi (CM) \ge \Psi (CO)$ follows immediately by comparing terms. By Lemma \ref{subgraph}, we have $\Psi (H_{n-2} - c) \ge \Psi (H_{n-2}-\{a,b,c,i,j\})$ and by comparing the remaining terms we see that $\Psi (CM) \ge \Psi (CP)$. The preceding shows that attaching a hexagon in the meta position yields the most maximal matchings, implying
$$
\Psi (G_n) \le \Psi (M_n)
$$
as desired. 

To get the remaining inequality of our theorem, we need only show that $\Psi (CO) \ge \Psi (CP)$. Now we must have either $(i)$ or $(ii)$ of Lemma \ref{subgraph2}, say $(i)$ holds. Then $4\cdot \Psi(H_{n-2}-\{b,c\}) \ge 2 \cdot \Psi (H_{n-2} - c)$ and by Lemma \ref{subgraph} we have $\Psi(H_{n-2}-\{c,i\}) \ge \Psi (H_{n-2}-\{a,b,c,i,j\})$, showing that
\begin{align}
\Psi (CO) &\ge 11\Psi(H_{n-2}-\{b,c\}) + 13\Psi(H_{n-2}-\{c,i\}) + 20\cdot \Psi (H_{n-2} - c) \notag \\
&+ 5\cdot \Psi (H_{n-2}-\{a,b,c,i,j\}). \label{ineqCO}
\end{align}
\noindent Now by comparing the terms of $\Psi (CP)$ with the inequality $(\ref{ineqCO})$, it follows that $\Psi (CO) \ge \Psi (CP)$, which completes the proof.
\end{proof}

It is instructive to compare the above results with the corresponding results
for all matchings and for independent sets from reference \cite{doslicmaloy} 
(Theorems 3.23 and 4.14, respectively). It can be seen that with respect 
to the richest chains, the number of
maximal matchings behaves more like the number of independent sets than the
number of all matchings. A possible explanation might be the fact that
maximal matchings in any graph $G$ are in a bijective correspondence with 
{\em nice} independent sets in $G$. (A set of vertices $S$ is {\em nice} if
$G-S$ has a perfect matching.)

\section{Benzenoid chains}

\subsection{Generating functions} \label{secbenzgen}

Now we turn our attention to benzenoid chains. Here the connectivity increases
to two, and one can expect that this will result in longer recurrences, as
indicated in \cite{doszub2}. This is, indeed, the case.

Using the same techniques outlined in subsection \ref{seccactigen}, we obtain ordinary generating functions for the number of maximal matchings in the benzenoid chains $L_n$, $Z_n$, and $H_n$. 

\begin{figure}[h]
\begin{center}
\begin{tikzpicture}[scale=.4]
\draw (-2.5,0) node {$L_n$};
\draw (.87,-.5) circle (2pt) [fill=black];
\draw (.87,.5) circle (2pt) [fill=black];
\draw (-.87,.5) circle (2pt) [fill=black];
\draw (-.87,-.5) circle (2pt) [fill=black];
\draw (0,1) circle (2pt) [fill=black];
\draw (0,-1) circle (2pt) [fill=black];
\draw (0,1)--(.87,.5);
\draw (.87,-.5)--(.87,.5);
\draw (.87,-.5)--(0,-1);
\draw (-.87,-.5)--(0,-1);
\draw (-.87,-.5)--(-.87,.5);
\draw (0,1)--(-.87,.5);

\draw (2.61,.5) circle (2pt) [fill=black];
\draw (2.61,-.5) circle (2pt) [fill=black];
\draw (1.74,1) circle (2pt) [fill=black];
\draw (1.74,-1) circle (2pt) [fill=black];
\draw (2.61,-.5)--(1.74,-1);
\draw (2.61,-.5)--(2.61,.5);
\draw (1.74,1)--(2.61,.5);
\draw (1.74,1)--(.87,.5);
\draw (1.74,-1)--(.87,-.5);

\draw (4.35,.5) circle (2pt) [fill=black];
\draw (4.35,-.5) circle (2pt) [fill=black];
\draw (3.48,1) circle (2pt) [fill=black];
\draw (3.48,-1) circle (2pt) [fill=black];
\draw (4.35,-.5)--(3.48,-1);
\draw (4.35,-.5)--(4.35,.5);
\draw (3.48,1)--(4.35,.5);
\draw (3.48,1)--(2.61,.5);
\draw (3.48,-1)--(2.61,-.5);

\draw (6.09,.5) circle (2pt) [fill=black];
\draw (6.09,-.5) circle (2pt) [fill=black];
\draw (5.22,1) circle (2pt) [fill=black];
\draw (5.22,-1) circle (2pt) [fill=black];
\draw (6.09,-.5)--(5.22,-1);
\draw (6.09,-.5)--(6.09,.5);
\draw (5.22,1)--(6.09,.5);
\draw (5.22,1)--(4.35,.5);
\draw (5.22,-1)--(4.35,-.5);

\draw (7.83,.5) circle (2pt) [fill=black];
\draw (7.83,-.5) circle (2pt) [fill=black];
\draw (6.96,1) circle (2pt) [fill=black];
\draw (6.96,-1) circle (2pt) [fill=black];
\draw (7.83,-.5)--(6.96,-1);
\draw (7.83,-.5)--(7.83,.5);
\draw (6.96,1)--(7.83,.5);
\draw (6.96,1)--(6.09,.5);
\draw (6.96,-1)--(6.09,-.5);

\draw (0,0) node {$1$};
\draw (1.74,0) node {$2$};
\draw (3.48,0) node {$\cdots$};
\draw (6.96,0) node {$n$};
\end{tikzpicture}
\hspace{.5cm}
\begin{tikzpicture}[scale=.4]
\draw (-2.5,0) node {$L^1_n$};
\draw (.87,-.5) circle (2pt) [fill=black];
\draw (.87,.5) circle (2pt) [fill=black];
\draw (-.87,.5) circle (2pt) [fill=black];
\draw (-.87,-.5) circle (2pt) [fill=black];
\draw (0,1) circle (2pt) [fill=black];
\draw (0,-1) circle (2pt) [fill=black];
\draw (0,1)--(.87,.5);
\draw (.87,-.5)--(.87,.5);
\draw (.87,-.5)--(0,-1);
\draw (-.87,-.5)--(0,-1);
\draw (-.87,-.5)--(-.87,.5);
\draw (0,1)--(-.87,.5);

\draw (2.61,.5) circle (2pt) [fill=black];
\draw (2.61,-.5) circle (2pt) [fill=black];
\draw (1.74,1) circle (2pt) [fill=black];
\draw (1.74,-1) circle (2pt) [fill=black];
\draw (2.61,-.5)--(1.74,-1);
\draw (2.61,-.5)--(2.61,.5);
\draw (1.74,1)--(2.61,.5);
\draw (1.74,1)--(.87,.5);
\draw (1.74,-1)--(.87,-.5);

\draw (4.35,.5) circle (2pt) [fill=black];
\draw (4.35,-.5) circle (2pt) [fill=black];
\draw (3.48,1) circle (2pt) [fill=black];
\draw (3.48,-1) circle (2pt) [fill=black];
\draw (4.35,-.5)--(3.48,-1);
\draw (4.35,-.5)--(4.35,.5);
\draw (3.48,1)--(4.35,.5);
\draw (3.48,1)--(2.61,.5);
\draw (3.48,-1)--(2.61,-.5);

\draw (6.09,.5) circle (2pt) [fill=black];
\draw (6.09,-.5) circle (2pt) [fill=black];
\draw (5.22,1) circle (2pt) [fill=black];
\draw (5.22,-1) circle (2pt) [fill=black];
\draw (6.09,-.5)--(5.22,-1);
\draw (6.09,-.5)--(6.09,.5);
\draw (5.22,1)--(6.09,.5);
\draw (5.22,1)--(4.35,.5);
\draw (5.22,-1)--(4.35,-.5);

\draw (7.83,.5) circle (2pt) [fill=black];
\draw (7.83,-.5) circle (2pt) [fill=black];
\draw (6.96,1) circle (2pt) [fill=black];
\draw (6.96,-1) circle (2pt) [fill=black];
\draw (7.83,-.5)--(6.96,-1);
\draw (7.83,-.5)--(7.83,.5);
\draw (6.96,1)--(7.83,.5);
\draw (6.96,1)--(6.09,.5);
\draw (6.96,-1)--(6.09,-.5);

\draw (8.7,1) circle (2pt) [fill=black];
\draw (8.7,-1) circle (2pt) [fill=black];
\draw (8.7,1)--(7.83,.5);
\draw (8.7,-1)--(7.83,-.5);

\draw (0,0) node {$1$};
\draw (1.74,0) node {$2$};
\draw (3.48,0) node {$\cdots$};
\draw (6.96,0) node {$n$};
\end{tikzpicture}

\vspace{.5cm}

\begin{tikzpicture}[scale=.4]
\draw (-2.5,0) node {$L^2_n$};
\draw (.87,-.5) circle (2pt) [fill=black];
\draw (.87,.5) circle (2pt) [fill=black];
\draw (-.87,.5) circle (2pt) [fill=black];
\draw (-.87,-.5) circle (2pt) [fill=black];
\draw (0,1) circle (2pt) [fill=black];
\draw (0,-1) circle (2pt) [fill=black];
\draw (0,1)--(.87,.5);
\draw (.87,-.5)--(.87,.5);
\draw (.87,-.5)--(0,-1);
\draw (-.87,-.5)--(0,-1);
\draw (-.87,-.5)--(-.87,.5);
\draw (0,1)--(-.87,.5);

\draw (2.61,.5) circle (2pt) [fill=black];
\draw (2.61,-.5) circle (2pt) [fill=black];
\draw (1.74,1) circle (2pt) [fill=black];
\draw (1.74,-1) circle (2pt) [fill=black];
\draw (2.61,-.5)--(1.74,-1);
\draw (2.61,-.5)--(2.61,.5);
\draw (1.74,1)--(2.61,.5);
\draw (1.74,1)--(.87,.5);
\draw (1.74,-1)--(.87,-.5);

\draw (4.35,.5) circle (2pt) [fill=black];
\draw (4.35,-.5) circle (2pt) [fill=black];
\draw (3.48,1) circle (2pt) [fill=black];
\draw (3.48,-1) circle (2pt) [fill=black];
\draw (4.35,-.5)--(3.48,-1);
\draw (4.35,-.5)--(4.35,.5);
\draw (3.48,1)--(4.35,.5);
\draw (3.48,1)--(2.61,.5);
\draw (3.48,-1)--(2.61,-.5);

\draw (6.09,.5) circle (2pt) [fill=black];
\draw (6.09,-.5) circle (2pt) [fill=black];
\draw (5.22,1) circle (2pt) [fill=black];
\draw (5.22,-1) circle (2pt) [fill=black];
\draw (6.09,-.5)--(5.22,-1);
\draw (6.09,-.5)--(6.09,.5);
\draw (5.22,1)--(6.09,.5);
\draw (5.22,1)--(4.35,.5);
\draw (5.22,-1)--(4.35,-.5);

\draw (7.83,.5) circle (2pt) [fill=black];
\draw (7.83,-.5) circle (2pt) [fill=black];
\draw (6.96,1) circle (2pt) [fill=black];
\draw (6.96,-1) circle (2pt) [fill=black];
\draw (7.83,-.5)--(6.96,-1);
\draw (7.83,-.5)--(7.83,.5);
\draw (6.96,1)--(7.83,.5);
\draw (6.96,1)--(6.09,.5);
\draw (6.96,-1)--(6.09,-.5);

\draw (9.57,.5) circle (2pt) [fill=black];
\draw (8.7,1) circle (2pt) [fill=black];
\draw (8.7,-1) circle (2pt) [fill=black];
\draw (8.7,1)--(9.57,.5);
\draw (8.7,1)--(7.83,.5);
\draw (8.7,-1)--(7.83,-.5);

\draw (0,0) node {$1$};
\draw (1.74,0) node {$2$};
\draw (3.48,0) node {$\cdots$};
\draw (6.96,0) node {$n$};
\end{tikzpicture}
\hspace{.5cm}
\begin{tikzpicture}[scale=.4]
\draw (-2.5,0) node {$L^3_n$};
\draw (.87,-.5) circle (2pt) [fill=black];
\draw (.87,.5) circle (2pt) [fill=black];
\draw (-.87,.5) circle (2pt) [fill=black];
\draw (-.87,-.5) circle (2pt) [fill=black];
\draw (0,1) circle (2pt) [fill=black];
\draw (0,-1) circle (2pt) [fill=black];
\draw (0,1)--(.87,.5);
\draw (.87,-.5)--(.87,.5);
\draw (.87,-.5)--(0,-1);
\draw (-.87,-.5)--(0,-1);
\draw (-.87,-.5)--(-.87,.5);
\draw (0,1)--(-.87,.5);

\draw (2.61,.5) circle (2pt) [fill=black];
\draw (2.61,-.5) circle (2pt) [fill=black];
\draw (1.74,1) circle (2pt) [fill=black];
\draw (1.74,-1) circle (2pt) [fill=black];
\draw (2.61,-.5)--(1.74,-1);
\draw (2.61,-.5)--(2.61,.5);
\draw (1.74,1)--(2.61,.5);
\draw (1.74,1)--(.87,.5);
\draw (1.74,-1)--(.87,-.5);

\draw (4.35,.5) circle (2pt) [fill=black];
\draw (4.35,-.5) circle (2pt) [fill=black];
\draw (3.48,1) circle (2pt) [fill=black];
\draw (3.48,-1) circle (2pt) [fill=black];
\draw (4.35,-.5)--(3.48,-1);
\draw (4.35,-.5)--(4.35,.5);
\draw (3.48,1)--(4.35,.5);
\draw (3.48,1)--(2.61,.5);
\draw (3.48,-1)--(2.61,-.5);

\draw (6.09,.5) circle (2pt) [fill=black];
\draw (6.09,-.5) circle (2pt) [fill=black];
\draw (5.22,1) circle (2pt) [fill=black];
\draw (5.22,-1) circle (2pt) [fill=black];
\draw (6.09,-.5)--(5.22,-1);
\draw (6.09,-.5)--(6.09,.5);
\draw (5.22,1)--(6.09,.5);
\draw (5.22,1)--(4.35,.5);
\draw (5.22,-1)--(4.35,-.5);

\draw (7.83,.5) circle (2pt) [fill=black];
\draw (7.83,-.5) circle (2pt) [fill=black];
\draw (6.96,1) circle (2pt) [fill=black];
\draw (6.96,-1) circle (2pt) [fill=black];
\draw (7.83,-.5)--(6.96,-1);
\draw (7.83,-.5)--(7.83,.5);
\draw (6.96,1)--(7.83,.5);
\draw (6.96,1)--(6.09,.5);
\draw (6.96,-1)--(6.09,-.5);

\draw (8.7,1) circle (2pt) [fill=black];
\draw (8.7,1)--(7.83,.5);

\draw (0,0) node {$1$};
\draw (1.74,0) node {$2$};
\draw (3.48,0) node {$\cdots$};
\draw (6.96,0) node {$n$};
\end{tikzpicture}

\end{center}
\caption{Auxiliary graphs for $L_n$.} \label{auxlinear}
\end{figure}
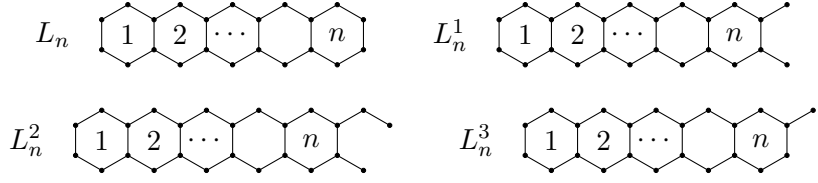

\begin{lemma} Let $\ell_n$ be the number of maximal matchings in $L_n$ and $\ell^i_n$ be the number of maximal matchings in the auxiliary graph $L^i_n$ in Figure \ref{auxlinear}. Then\\

$(i)$ $\ell _n = \ell ^1 _{n-1} + \ell _{n-1} + 2\ell ^2 _{n-2}$, \\

$(ii)$ $\ell ^1 _n = 2 \ell ^1 _{n-1} + \ell _{n-1} + 2\ell ^3 _{n-1}$, \\

$(iii)$ $\ell ^2 _n = \ell ^3 _n + \ell ^1 _{n-1} + \ell ^3 _{n-1}$, \\

$(iv)$ $\ell ^3 _n = \ell ^1 _{n-1} + \ell _{n-1} + \ell ^3 _{n-1} + \ell ^2 _{n-2} + \ell ^1 _{n-2} + \ell ^3 _{n-2}$,\\

\noindent with the initial conditions $\ell_0=1$, $\ell_1=5$, $\ell ^1_0=2$, $\ell ^2_0=3$, $\ell ^3_0=2$, and $\ell ^3_1=7$.
\end{lemma}

\begin{figure}[h]
\begin{center}
\begin{tikzpicture}[scale=.4]
\draw (-2.5,0) node {$Z_n$};
\draw (-.5,.87) circle (2pt) [fill=black];
\draw (.5,.87) circle (2pt) [fill=black];
\draw (.5,-.87) circle (2pt) [fill=black];
\draw (-.5,-.87) circle (2pt) [fill=black];
\draw (1,0) circle (2pt) [fill=black];
\draw (-1,0) circle (2pt) [fill=black];
\draw (1,0)--(.5,.87);
\draw (-.5,.87)--(.5,.87);
\draw (-.5,.87)--(-1,0);
\draw (-.5,-.87)--(-1,0);
\draw (-.5,-.87)--(.5,-.87);
\draw (1,0)--(.5,-.87);

\draw (2,0) circle (2pt) [fill=black];
\draw (2.5,-.87) circle (2pt) [fill=black];
\draw (1,-1.74) circle (2pt) [fill=black];
\draw (2,-1.74) circle (2pt) [fill=black];
\draw (1,0)--(2,0);
\draw (2.5,-.87)--(2,0);
\draw (2.5,-.87)--(2,-1.74);
\draw (1,-1.74)--(2,-1.74);
\draw (1,-1.74)--(.5,-.87);

\draw (2.5,.87) circle (2pt) [fill=black];
\draw (3.5,.87) circle (2pt) [fill=black];
\draw (3.5,-.87) circle (2pt) [fill=black];
\draw (4,0) circle (2pt) [fill=black];
\draw (4,0)--(3.5,.87);
\draw (2.5,.87)--(3.5,.87);
\draw (2.5,.87)--(2,0);
\draw (2.5,-.87)--(3.5,-.87);
\draw (4,0)--(3.5,-.87);

\draw (5,0) circle (2pt) [fill=black];
\draw (5.5,-.87) circle (2pt) [fill=black];
\draw (4,-1.74) circle (2pt) [fill=black];
\draw (5,-1.74) circle (2pt) [fill=black];
\draw (4,0)--(5,0);
\draw (5.5,-.87)--(5,0);
\draw (5.5,-.87)--(5,-1.74);
\draw (4,-1.74)--(5,-1.74);
\draw (4,-1.74)--(3.5,-.87);

\draw (5.5,.87) circle (2pt) [fill=black];
\draw (6.5,.87) circle (2pt) [fill=black];
\draw (6.5,-.87) circle (2pt) [fill=black];
\draw (7,0) circle (2pt) [fill=black];
\draw (7,0)--(6.5,.87);
\draw (5.5,.87)--(6.5,.87);
\draw (5.5,.87)--(5,0);
\draw (5.5,-.87)--(6.5,-.87);
\draw (7,0)--(6.5,-.87);

\draw (0,0) node {$1$};
\draw (1.5,-.87) node {$2$};
\draw (3,0) node {$\cdots$};
\draw (6,0) node {$n$};
\end{tikzpicture}
\hspace{.5cm}
\begin{tikzpicture}[scale=.4]
\draw (-2.5,0) node {$Z^1_n$};
\draw (-.5,.87) circle (2pt) [fill=black];
\draw (.5,.87) circle (2pt) [fill=black];
\draw (.5,-.87) circle (2pt) [fill=black];
\draw (-.5,-.87) circle (2pt) [fill=black];
\draw (1,0) circle (2pt) [fill=black];
\draw (-1,0) circle (2pt) [fill=black];
\draw (1,0)--(.5,.87);
\draw (-.5,.87)--(.5,.87);
\draw (-.5,.87)--(-1,0);
\draw (-.5,-.87)--(-1,0);
\draw (-.5,-.87)--(.5,-.87);
\draw (1,0)--(.5,-.87);

\draw (2,0) circle (2pt) [fill=black];
\draw (2.5,-.87) circle (2pt) [fill=black];
\draw (1,-1.74) circle (2pt) [fill=black];
\draw (2,-1.74) circle (2pt) [fill=black];
\draw (1,0)--(2,0);
\draw (2.5,-.87)--(2,0);
\draw (2.5,-.87)--(2,-1.74);
\draw (1,-1.74)--(2,-1.74);
\draw (1,-1.74)--(.5,-.87);

\draw (2.5,.87) circle (2pt) [fill=black];
\draw (3.5,.87) circle (2pt) [fill=black];
\draw (3.5,-.87) circle (2pt) [fill=black];
\draw (4,0) circle (2pt) [fill=black];
\draw (4,0)--(3.5,.87);
\draw (2.5,.87)--(3.5,.87);
\draw (2.5,.87)--(2,0);
\draw (2.5,-.87)--(3.5,-.87);
\draw (4,0)--(3.5,-.87);

\draw (5,0) circle (2pt) [fill=black];
\draw (5.5,-.87) circle (2pt) [fill=black];
\draw (4,-1.74) circle (2pt) [fill=black];
\draw (5,-1.74) circle (2pt) [fill=black];
\draw (4,0)--(5,0);
\draw (5.5,-.87)--(5,0);
\draw (5.5,-.87)--(5,-1.74);
\draw (4,-1.74)--(5,-1.74);
\draw (4,-1.74)--(3.5,-.87);

\draw (5.5,.87) circle (2pt) [fill=black];
\draw (6.5,.87) circle (2pt) [fill=black];
\draw (6.5,-.87) circle (2pt) [fill=black];
\draw (7,0) circle (2pt) [fill=black];
\draw (7,0)--(6.5,.87);
\draw (5.5,.87)--(6.5,.87);
\draw (5.5,.87)--(5,0);
\draw (5.5,-.87)--(6.5,-.87);
\draw (7,0)--(6.5,-.87);

\draw (8,0) circle (2pt) [fill=black];
\draw (7,-1.74) circle (2pt) [fill=black];
\draw (7,0)--(8,0);
\draw (7,-1.74)--(6.5,-.87);

\draw (0,0) node {$1$};
\draw (1.5,-.87) node {$2$};
\draw (3,0) node {$\cdots$};
\draw (6,0) node {$n$};
\end{tikzpicture}

\vspace{.5cm}

\begin{tikzpicture}[scale=.4]
\draw (-2.5,0) node {$Z^2_n$};
\draw (-.5,.87) circle (2pt) [fill=black];
\draw (.5,.87) circle (2pt) [fill=black];
\draw (.5,-.87) circle (2pt) [fill=black];
\draw (-.5,-.87) circle (2pt) [fill=black];
\draw (1,0) circle (2pt) [fill=black];
\draw (-1,0) circle (2pt) [fill=black];
\draw (1,0)--(.5,.87);
\draw (-.5,.87)--(.5,.87);
\draw (-.5,.87)--(-1,0);
\draw (-.5,-.87)--(-1,0);
\draw (-.5,-.87)--(.5,-.87);
\draw (1,0)--(.5,-.87);

\draw (2,0) circle (2pt) [fill=black];
\draw (2.5,-.87) circle (2pt) [fill=black];
\draw (1,-1.74) circle (2pt) [fill=black];
\draw (2,-1.74) circle (2pt) [fill=black];
\draw (1,0)--(2,0);
\draw (2.5,-.87)--(2,0);
\draw (2.5,-.87)--(2,-1.74);
\draw (1,-1.74)--(2,-1.74);
\draw (1,-1.74)--(.5,-.87);

\draw (2.5,.87) circle (2pt) [fill=black];
\draw (3.5,.87) circle (2pt) [fill=black];
\draw (3.5,-.87) circle (2pt) [fill=black];
\draw (4,0) circle (2pt) [fill=black];
\draw (4,0)--(3.5,.87);
\draw (2.5,.87)--(3.5,.87);
\draw (2.5,.87)--(2,0);
\draw (2.5,-.87)--(3.5,-.87);
\draw (4,0)--(3.5,-.87);

\draw (5,0) circle (2pt) [fill=black];
\draw (5.5,-.87) circle (2pt) [fill=black];
\draw (4,-1.74) circle (2pt) [fill=black];
\draw (5,-1.74) circle (2pt) [fill=black];
\draw (4,0)--(5,0);
\draw (5.5,-.87)--(5,0);
\draw (5.5,-.87)--(5,-1.74);
\draw (4,-1.74)--(5,-1.74);
\draw (4,-1.74)--(3.5,-.87);

\draw (5.5,.87) circle (2pt) [fill=black];
\draw (6.5,.87) circle (2pt) [fill=black];
\draw (6.5,-.87) circle (2pt) [fill=black];
\draw (7,0) circle (2pt) [fill=black];
\draw (7,0)--(6.5,.87);
\draw (5.5,.87)--(6.5,.87);
\draw (5.5,.87)--(5,0);
\draw (5.5,-.87)--(6.5,-.87);
\draw (7,0)--(6.5,-.87);

\draw (7,-1.74) circle (2pt) [fill=black];
\draw (8,-1.74) circle (2pt) [fill=black];
\draw (7,-1.74)--(8,-1.74);
\draw (7,-1.74)--(6.5,-.87);

\draw (0,0) node {$1$};
\draw (1.5,-.87) node {$2$};
\draw (3,0) node {$\cdots$};
\draw (6,0) node {$n$};
\end{tikzpicture}
\hspace{.5cm}
\begin{tikzpicture}[scale=.4]
\draw (-2.5,0) node {$Z^3_n$};
\draw (-.5,.87) circle (2pt) [fill=black];
\draw (.5,.87) circle (2pt) [fill=black];
\draw (.5,-.87) circle (2pt) [fill=black];
\draw (-.5,-.87) circle (2pt) [fill=black];
\draw (1,0) circle (2pt) [fill=black];
\draw (-1,0) circle (2pt) [fill=black];
\draw (1,0)--(.5,.87);
\draw (-.5,.87)--(.5,.87);
\draw (-.5,.87)--(-1,0);
\draw (-.5,-.87)--(-1,0);
\draw (-.5,-.87)--(.5,-.87);
\draw (1,0)--(.5,-.87);

\draw (2,0) circle (2pt) [fill=black];
\draw (2.5,-.87) circle (2pt) [fill=black];
\draw (1,-1.74) circle (2pt) [fill=black];
\draw (2,-1.74) circle (2pt) [fill=black];
\draw (1,0)--(2,0);
\draw (2.5,-.87)--(2,0);
\draw (2.5,-.87)--(2,-1.74);
\draw (1,-1.74)--(2,-1.74);
\draw (1,-1.74)--(.5,-.87);

\draw (2.5,.87) circle (2pt) [fill=black];
\draw (3.5,.87) circle (2pt) [fill=black];
\draw (3.5,-.87) circle (2pt) [fill=black];
\draw (4,0) circle (2pt) [fill=black];
\draw (4,0)--(3.5,.87);
\draw (2.5,.87)--(3.5,.87);
\draw (2.5,.87)--(2,0);
\draw (2.5,-.87)--(3.5,-.87);
\draw (4,0)--(3.5,-.87);

\draw (5,0) circle (2pt) [fill=black];
\draw (5.5,-.87) circle (2pt) [fill=black];
\draw (4,-1.74) circle (2pt) [fill=black];
\draw (5,-1.74) circle (2pt) [fill=black];
\draw (4,0)--(5,0);
\draw (5.5,-.87)--(5,0);
\draw (5.5,-.87)--(5,-1.74);
\draw (4,-1.74)--(5,-1.74);
\draw (4,-1.74)--(3.5,-.87);

\draw (5.5,.87) circle (2pt) [fill=black];
\draw (6.5,.87) circle (2pt) [fill=black];
\draw (6.5,-.87) circle (2pt) [fill=black];
\draw (7,0) circle (2pt) [fill=black];
\draw (7,0)--(6.5,.87);
\draw (5.5,.87)--(6.5,.87);
\draw (5.5,.87)--(5,0);
\draw (5.5,-.87)--(6.5,-.87);
\draw (7,0)--(6.5,-.87);

\draw (8,0) circle (2pt) [fill=black];
\draw (7,-1.74) circle (2pt) [fill=black];
\draw (8,-1.74) circle (2pt) [fill=black];
\draw (7,0)--(8,0);
\draw (7,-1.74)--(8,-1.74);
\draw (7,-1.74)--(6.5,-.87);

\draw (0,0) node {$1$};
\draw (1.5,-.87) node {$2$};
\draw (3,0) node {$\cdots$};
\draw (6,0) node {$n$};
\end{tikzpicture}

\vspace{.5cm}

\begin{tikzpicture}[scale=.4]
\draw (-2.5,0) node {$Z^4_n$};
\draw (-.5,.87) circle (2pt) [fill=black];
\draw (.5,.87) circle (2pt) [fill=black];
\draw (.5,-.87) circle (2pt) [fill=black];
\draw (-.5,-.87) circle (2pt) [fill=black];
\draw (1,0) circle (2pt) [fill=black];
\draw (-1,0) circle (2pt) [fill=black];
\draw (1,0)--(.5,.87);
\draw (-.5,.87)--(.5,.87);
\draw (-.5,.87)--(-1,0);
\draw (-.5,-.87)--(-1,0);
\draw (-.5,-.87)--(.5,-.87);
\draw (1,0)--(.5,-.87);

\draw (2,0) circle (2pt) [fill=black];
\draw (2.5,-.87) circle (2pt) [fill=black];
\draw (1,-1.74) circle (2pt) [fill=black];
\draw (2,-1.74) circle (2pt) [fill=black];
\draw (1,0)--(2,0);
\draw (2.5,-.87)--(2,0);
\draw (2.5,-.87)--(2,-1.74);
\draw (1,-1.74)--(2,-1.74);
\draw (1,-1.74)--(.5,-.87);

\draw (2.5,.87) circle (2pt) [fill=black];
\draw (3.5,.87) circle (2pt) [fill=black];
\draw (3.5,-.87) circle (2pt) [fill=black];
\draw (4,0) circle (2pt) [fill=black];
\draw (4,0)--(3.5,.87);
\draw (2.5,.87)--(3.5,.87);
\draw (2.5,.87)--(2,0);
\draw (2.5,-.87)--(3.5,-.87);
\draw (4,0)--(3.5,-.87);

\draw (5,0) circle (2pt) [fill=black];
\draw (5.5,-.87) circle (2pt) [fill=black];
\draw (4,-1.74) circle (2pt) [fill=black];
\draw (5,-1.74) circle (2pt) [fill=black];
\draw (4,0)--(5,0);
\draw (5.5,-.87)--(5,0);
\draw (5.5,-.87)--(5,-1.74);
\draw (4,-1.74)--(5,-1.74);
\draw (4,-1.74)--(3.5,-.87);

\draw (5.5,.87) circle (2pt) [fill=black];
\draw (6.5,.87) circle (2pt) [fill=black];
\draw (6.5,-.87) circle (2pt) [fill=black];
\draw (7,0) circle (2pt) [fill=black];
\draw (7,0)--(6.5,.87);
\draw (5.5,.87)--(6.5,.87);
\draw (5.5,.87)--(5,0);
\draw (5.5,-.87)--(6.5,-.87);
\draw (7,0)--(6.5,-.87);

\draw (8,0) circle (2pt) [fill=black];
\draw (7,-1.74) circle (2pt) [fill=black];
\draw (8,-1.74) circle (2pt) [fill=black];
\draw (8.5,.87) circle (2pt) [fill=black];
\draw (7,0)--(8,0);
\draw (8.5,.87)--(8,0);
\draw (7,-1.74)--(8,-1.74);
\draw (7,-1.74)--(6.5,-.87);

\draw (0,0) node {$1$};
\draw (1.5,-.87) node {$2$};
\draw (3,0) node {$\cdots$};
\draw (6,0) node {$n$};
\end{tikzpicture}
\hspace{.5cm}
\begin{tikzpicture}[scale=.4]
\draw (-2.5,0) node {$Z^5_n$};
\draw (-.5,.87) circle (2pt) [fill=black];
\draw (.5,.87) circle (2pt) [fill=black];
\draw (.5,-.87) circle (2pt) [fill=black];
\draw (-.5,-.87) circle (2pt) [fill=black];
\draw (1,0) circle (2pt) [fill=black];
\draw (-1,0) circle (2pt) [fill=black];
\draw (1,0)--(.5,.87);
\draw (-.5,.87)--(.5,.87);
\draw (-.5,.87)--(-1,0);
\draw (-.5,-.87)--(-1,0);
\draw (-.5,-.87)--(.5,-.87);
\draw (1,0)--(.5,-.87);

\draw (2,0) circle (2pt) [fill=black];
\draw (2.5,-.87) circle (2pt) [fill=black];
\draw (1,-1.74) circle (2pt) [fill=black];
\draw (2,-1.74) circle (2pt) [fill=black];
\draw (1,0)--(2,0);
\draw (2.5,-.87)--(2,0);
\draw (2.5,-.87)--(2,-1.74);
\draw (1,-1.74)--(2,-1.74);
\draw (1,-1.74)--(.5,-.87);

\draw (2.5,.87) circle (2pt) [fill=black];
\draw (3.5,.87) circle (2pt) [fill=black];
\draw (3.5,-.87) circle (2pt) [fill=black];
\draw (4,0) circle (2pt) [fill=black];
\draw (4,0)--(3.5,.87);
\draw (2.5,.87)--(3.5,.87);
\draw (2.5,.87)--(2,0);
\draw (2.5,-.87)--(3.5,-.87);
\draw (4,0)--(3.5,-.87);

\draw (5,0) circle (2pt) [fill=black];
\draw (5.5,-.87) circle (2pt) [fill=black];
\draw (4,-1.74) circle (2pt) [fill=black];
\draw (5,-1.74) circle (2pt) [fill=black];
\draw (4,0)--(5,0);
\draw (5.5,-.87)--(5,0);
\draw (5.5,-.87)--(5,-1.74);
\draw (4,-1.74)--(5,-1.74);
\draw (4,-1.74)--(3.5,-.87);

\draw (5.5,.87) circle (2pt) [fill=black];
\draw (6.5,.87) circle (2pt) [fill=black];
\draw (6.5,-.87) circle (2pt) [fill=black];
\draw (7,0) circle (2pt) [fill=black];
\draw (7,0)--(6.5,.87);
\draw (5.5,.87)--(6.5,.87);
\draw (5.5,.87)--(5,0);
\draw (5.5,-.87)--(6.5,-.87);
\draw (7,0)--(6.5,-.87);

\draw (7,-1.74) circle (2pt) [fill=black];
\draw (7,-1.74)--(6.5,-.87);

\draw (0,0) node {$1$};
\draw (1.5,-.87) node {$2$};
\draw (3,0) node {$\cdots$};
\draw (6,0) node {$n$};
\end{tikzpicture}

\end{center}
\caption{Auxiliary graphs for $Z_n$.} \label{auxzigzag}
\end{figure}
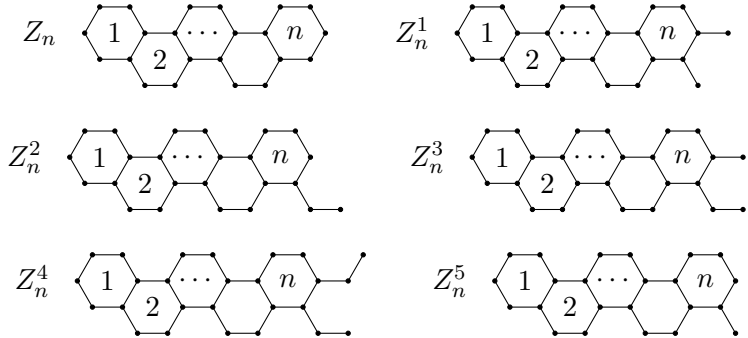

\begin{lemma} Let $z_n$ be the number of maximal matchings in $Z_n$ and $z^i_n$ be the number of maximal matchings in the auxiliary graph $Z^i_n$ in Figure \ref{auxzigzag}. Then\\

$(i)$ $z _n = z^1_{n-1} + z^2 _{n-1} + z^3_{n-2}$, \\

$(ii)$ $z ^1 _n = 2z^2_{n-1} + z^4_{n-2} + z^5_{n-1} + z^3_{n-2} + z^2_{n-2}$, \\

$(iii)$ $z ^2 _n = z_n + z^5_{n-1} + z_{n-1}$, \\

$(iv)$ $z ^3 _n = 2z^2_{n-1} + z^3_{n-1} + z^1_{n-1} + z^5_{n-1}$,\\

$(v)$ $z ^4 _n = z_n + z^5_{n-1} + z_{n-1} + z^2_{n-1} + z^3_{n-1}$,\\

$(vi)$ $z ^5 _n = z^5_{n-1} + z^4_{n-2} + z^2_{n-1} + z^3_{n-2} + z_{n-1}$,\\

\noindent with the initial conditions $z_0=1$, $z_1=5$, $z^1_0=2$, $z^1_1=9$, $z^2_0=2$, $z^3_0=3$, $z^4_0=4$, $z^5_0=2$, and $z^5_1=7$.
\end{lemma}

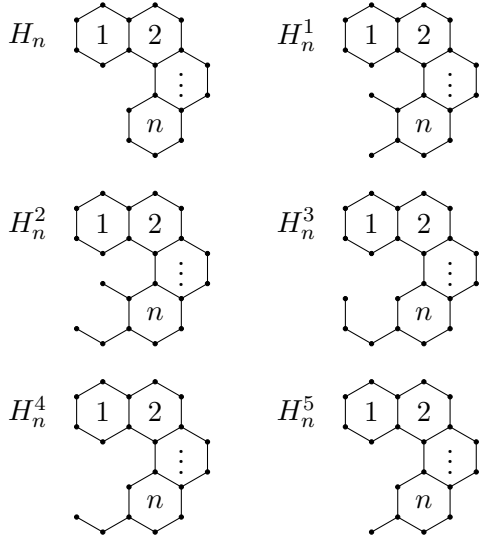
\begin{figure}[h]
\begin{center}
\begin{tikzpicture}[scale=.4]
\draw (-2.5,0) node {$H_n$};
\draw (.87,-.5) circle (2pt) [fill=black];
\draw (.87,.5) circle (2pt) [fill=black];
\draw (-.87,.5) circle (2pt) [fill=black];
\draw (-.87,-.5) circle (2pt) [fill=black];
\draw (0,1) circle (2pt) [fill=black];
\draw (0,-1) circle (2pt) [fill=black];
\draw (0,1)--(.87,.5);
\draw (.87,-.5)--(.87,.5);
\draw (.87,-.5)--(0,-1);
\draw (-.87,-.5)--(0,-1);
\draw (-.87,-.5)--(-.87,.5);
\draw (0,1)--(-.87,.5);

\draw (2.61,.5) circle (2pt) [fill=black];
\draw (2.61,-.5) circle (2pt) [fill=black];
\draw (1.74,1) circle (2pt) [fill=black];
\draw (1.74,-1) circle (2pt) [fill=black];
\draw (2.61,-.5)--(1.74,-1);
\draw (2.61,-.5)--(2.61,.5);
\draw (1.74,1)--(2.61,.5);
\draw (1.74,1)--(.87,.5);
\draw (1.74,-1)--(.87,-.5);

\draw (3.48,-1) circle (2pt) [fill=black];
\draw (3.48,-2) circle (2pt) [fill=black];
\draw (2.61,-2.5) circle (2pt) [fill=black];
\draw (1.74,-2) circle (2pt) [fill=black];
\draw (1.74,-2)--(1.74,-1);
\draw (3.48,-2)--(2.61,-2.5);
\draw (3.48,-2)--(3.48,-1);
\draw (2.61,-.5)--(3.48,-1);
\draw (2.61,-2.5)--(1.74,-2);

\draw (2.61,-3.5) circle (2pt) [fill=black];
\draw (1.74,-4) circle (2pt) [fill=black];
\draw (.87,-2.5) circle (2pt) [fill=black];
\draw (.87,-3.5) circle (2pt) [fill=black];
\draw (.87,-2.5)--(.87,-3.5);
\draw (2.61,-3.5)--(1.74,-4);
\draw (2.61,-3.5)--(2.61,-2.5);
\draw (1.74,-2)--(.87,-2.5);
\draw (1.74,-4)--(.87,-3.5);

\draw (0,0) node {$1$};
\draw (1.74,0) node {$2$};
\draw (2.61,-1.4) node {$\vdots$};
\draw (1.74,-3) node {$n$};
\end{tikzpicture}
\hspace{.5cm}
\begin{tikzpicture}[scale=.4]
\draw (-2.5,0) node {$H^1_n$};
\draw (.87,-.5) circle (2pt) [fill=black];
\draw (.87,.5) circle (2pt) [fill=black];
\draw (-.87,.5) circle (2pt) [fill=black];
\draw (-.87,-.5) circle (2pt) [fill=black];
\draw (0,1) circle (2pt) [fill=black];
\draw (0,-1) circle (2pt) [fill=black];
\draw (0,1)--(.87,.5);
\draw (.87,-.5)--(.87,.5);
\draw (.87,-.5)--(0,-1);
\draw (-.87,-.5)--(0,-1);
\draw (-.87,-.5)--(-.87,.5);
\draw (0,1)--(-.87,.5);

\draw (2.61,.5) circle (2pt) [fill=black];
\draw (2.61,-.5) circle (2pt) [fill=black];
\draw (1.74,1) circle (2pt) [fill=black];
\draw (1.74,-1) circle (2pt) [fill=black];
\draw (2.61,-.5)--(1.74,-1);
\draw (2.61,-.5)--(2.61,.5);
\draw (1.74,1)--(2.61,.5);
\draw (1.74,1)--(.87,.5);
\draw (1.74,-1)--(.87,-.5);

\draw (3.48,-1) circle (2pt) [fill=black];
\draw (3.48,-2) circle (2pt) [fill=black];
\draw (2.61,-2.5) circle (2pt) [fill=black];
\draw (1.74,-2) circle (2pt) [fill=black];
\draw (1.74,-2)--(1.74,-1);
\draw (3.48,-2)--(2.61,-2.5);
\draw (3.48,-2)--(3.48,-1);
\draw (2.61,-.5)--(3.48,-1);
\draw (2.61,-2.5)--(1.74,-2);

\draw (2.61,-3.5) circle (2pt) [fill=black];
\draw (1.74,-4) circle (2pt) [fill=black];
\draw (.87,-2.5) circle (2pt) [fill=black];
\draw (.87,-3.5) circle (2pt) [fill=black];
\draw (.87,-2.5)--(.87,-3.5);
\draw (2.61,-3.5)--(1.74,-4);
\draw (2.61,-3.5)--(2.61,-2.5);
\draw (1.74,-2)--(.87,-2.5);
\draw (1.74,-4)--(.87,-3.5);

\draw (0,-2) circle (2pt) [fill=black];
\draw (0,-4) circle (2pt) [fill=black];
\draw (0,-2)--(.87,-2.5);
\draw (.87,-3.5)--(0,-4);

\draw (0,0) node {$1$};
\draw (1.74,0) node {$2$};
\draw (2.61,-1.4) node {$\vdots$};
\draw (1.74,-3) node {$n$};
\end{tikzpicture}

\vspace{.4cm}

\begin{tikzpicture}[scale=.4]
\draw (-2.5,0) node {$H^2_n$};
\draw (.87,-.5) circle (2pt) [fill=black];
\draw (.87,.5) circle (2pt) [fill=black];
\draw (-.87,.5) circle (2pt) [fill=black];
\draw (-.87,-.5) circle (2pt) [fill=black];
\draw (0,1) circle (2pt) [fill=black];
\draw (0,-1) circle (2pt) [fill=black];
\draw (0,1)--(.87,.5);
\draw (.87,-.5)--(.87,.5);
\draw (.87,-.5)--(0,-1);
\draw (-.87,-.5)--(0,-1);
\draw (-.87,-.5)--(-.87,.5);
\draw (0,1)--(-.87,.5);

\draw (2.61,.5) circle (2pt) [fill=black];
\draw (2.61,-.5) circle (2pt) [fill=black];
\draw (1.74,1) circle (2pt) [fill=black];
\draw (1.74,-1) circle (2pt) [fill=black];
\draw (2.61,-.5)--(1.74,-1);
\draw (2.61,-.5)--(2.61,.5);
\draw (1.74,1)--(2.61,.5);
\draw (1.74,1)--(.87,.5);
\draw (1.74,-1)--(.87,-.5);

\draw (3.48,-1) circle (2pt) [fill=black];
\draw (3.48,-2) circle (2pt) [fill=black];
\draw (2.61,-2.5) circle (2pt) [fill=black];
\draw (1.74,-2) circle (2pt) [fill=black];
\draw (1.74,-2)--(1.74,-1);
\draw (3.48,-2)--(2.61,-2.5);
\draw (3.48,-2)--(3.48,-1);
\draw (2.61,-.5)--(3.48,-1);
\draw (2.61,-2.5)--(1.74,-2);

\draw (2.61,-3.5) circle (2pt) [fill=black];
\draw (1.74,-4) circle (2pt) [fill=black];
\draw (.87,-2.5) circle (2pt) [fill=black];
\draw (.87,-3.5) circle (2pt) [fill=black];
\draw (.87,-2.5)--(.87,-3.5);
\draw (2.61,-3.5)--(1.74,-4);
\draw (2.61,-3.5)--(2.61,-2.5);
\draw (1.74,-2)--(.87,-2.5);
\draw (1.74,-4)--(.87,-3.5);

\draw (-.87,-3.5) circle (2pt) [fill=black];
\draw (0,-2) circle (2pt) [fill=black];
\draw (0,-4) circle (2pt) [fill=black];
\draw (0,-2)--(.87,-2.5);
\draw (.87,-3.5)--(0,-4);
\draw (-.87,-3.5)--(0,-4);

\draw (0,0) node {$1$};
\draw (1.74,0) node {$2$};
\draw (2.61,-1.4) node {$\vdots$};
\draw (1.74,-3) node {$n$};
\end{tikzpicture}
\hspace{.5cm}
\begin{tikzpicture}[scale=.4]
\draw (-2.5,0) node {$H^3_n$};
\draw (.87,-.5) circle (2pt) [fill=black];
\draw (.87,.5) circle (2pt) [fill=black];
\draw (-.87,.5) circle (2pt) [fill=black];
\draw (-.87,-.5) circle (2pt) [fill=black];
\draw (0,1) circle (2pt) [fill=black];
\draw (0,-1) circle (2pt) [fill=black];
\draw (0,1)--(.87,.5);
\draw (.87,-.5)--(.87,.5);
\draw (.87,-.5)--(0,-1);
\draw (-.87,-.5)--(0,-1);
\draw (-.87,-.5)--(-.87,.5);
\draw (0,1)--(-.87,.5);

\draw (2.61,.5) circle (2pt) [fill=black];
\draw (2.61,-.5) circle (2pt) [fill=black];
\draw (1.74,1) circle (2pt) [fill=black];
\draw (1.74,-1) circle (2pt) [fill=black];
\draw (2.61,-.5)--(1.74,-1);
\draw (2.61,-.5)--(2.61,.5);
\draw (1.74,1)--(2.61,.5);
\draw (1.74,1)--(.87,.5);
\draw (1.74,-1)--(.87,-.5);

\draw (3.48,-1) circle (2pt) [fill=black];
\draw (3.48,-2) circle (2pt) [fill=black];
\draw (2.61,-2.5) circle (2pt) [fill=black];
\draw (1.74,-2) circle (2pt) [fill=black];
\draw (1.74,-2)--(1.74,-1);
\draw (3.48,-2)--(2.61,-2.5);
\draw (3.48,-2)--(3.48,-1);
\draw (2.61,-.5)--(3.48,-1);
\draw (2.61,-2.5)--(1.74,-2);

\draw (2.61,-3.5) circle (2pt) [fill=black];
\draw (1.74,-4) circle (2pt) [fill=black];
\draw (.87,-2.5) circle (2pt) [fill=black];
\draw (.87,-3.5) circle (2pt) [fill=black];
\draw (.87,-2.5)--(.87,-3.5);
\draw (2.61,-3.5)--(1.74,-4);
\draw (2.61,-3.5)--(2.61,-2.5);
\draw (1.74,-2)--(.87,-2.5);
\draw (1.74,-4)--(.87,-3.5);

\draw (-.87,-2.5) circle (2pt) [fill=black];
\draw (-.87,-3.5) circle (2pt) [fill=black];
\draw (0,-4) circle (2pt) [fill=black];
\draw (.87,-3.5)--(0,-4);
\draw (-.87,-3.5)--(0,-4);
\draw (-.87,-3.5)--(-.87,-2.5);

\draw (0,0) node {$1$};
\draw (1.74,0) node {$2$};
\draw (2.61,-1.4) node {$\vdots$};
\draw (1.74,-3) node {$n$};
\end{tikzpicture}

\vspace{.4cm}

\begin{tikzpicture}[scale=.4]
\draw (-2.5,0) node {$H^4_n$};
\draw (.87,-.5) circle (2pt) [fill=black];
\draw (.87,.5) circle (2pt) [fill=black];
\draw (-.87,.5) circle (2pt) [fill=black];
\draw (-.87,-.5) circle (2pt) [fill=black];
\draw (0,1) circle (2pt) [fill=black];
\draw (0,-1) circle (2pt) [fill=black];
\draw (0,1)--(.87,.5);
\draw (.87,-.5)--(.87,.5);
\draw (.87,-.5)--(0,-1);
\draw (-.87,-.5)--(0,-1);
\draw (-.87,-.5)--(-.87,.5);
\draw (0,1)--(-.87,.5);

\draw (2.61,.5) circle (2pt) [fill=black];
\draw (2.61,-.5) circle (2pt) [fill=black];
\draw (1.74,1) circle (2pt) [fill=black];
\draw (1.74,-1) circle (2pt) [fill=black];
\draw (2.61,-.5)--(1.74,-1);
\draw (2.61,-.5)--(2.61,.5);
\draw (1.74,1)--(2.61,.5);
\draw (1.74,1)--(.87,.5);
\draw (1.74,-1)--(.87,-.5);

\draw (3.48,-1) circle (2pt) [fill=black];
\draw (3.48,-2) circle (2pt) [fill=black];
\draw (2.61,-2.5) circle (2pt) [fill=black];
\draw (1.74,-2) circle (2pt) [fill=black];
\draw (1.74,-2)--(1.74,-1);
\draw (3.48,-2)--(2.61,-2.5);
\draw (3.48,-2)--(3.48,-1);
\draw (2.61,-.5)--(3.48,-1);
\draw (2.61,-2.5)--(1.74,-2);

\draw (2.61,-3.5) circle (2pt) [fill=black];
\draw (1.74,-4) circle (2pt) [fill=black];
\draw (.87,-2.5) circle (2pt) [fill=black];
\draw (.87,-3.5) circle (2pt) [fill=black];
\draw (.87,-2.5)--(.87,-3.5);
\draw (2.61,-3.5)--(1.74,-4);
\draw (2.61,-3.5)--(2.61,-2.5);
\draw (1.74,-2)--(.87,-2.5);
\draw (1.74,-4)--(.87,-3.5);

\draw (-.87,-3.5) circle (2pt) [fill=black];
\draw (0,-4) circle (2pt) [fill=black];
\draw (.87,-3.5)--(0,-4);
\draw (-.87,-3.5)--(0,-4);

\draw (0,0) node {$1$};
\draw (1.74,0) node {$2$};
\draw (2.61,-1.4) node {$\vdots$};
\draw (1.74,-3) node {$n$};
\end{tikzpicture}
\hspace{.5cm}
\begin{tikzpicture}[scale=.4]
\draw (-2.5,0) node {$H^5_n$};
\draw (.87,-.5) circle (2pt) [fill=black];
\draw (.87,.5) circle (2pt) [fill=black];
\draw (-.87,.5) circle (2pt) [fill=black];
\draw (-.87,-.5) circle (2pt) [fill=black];
\draw (0,1) circle (2pt) [fill=black];
\draw (0,-1) circle (2pt) [fill=black];
\draw (0,1)--(.87,.5);
\draw (.87,-.5)--(.87,.5);
\draw (.87,-.5)--(0,-1);
\draw (-.87,-.5)--(0,-1);
\draw (-.87,-.5)--(-.87,.5);
\draw (0,1)--(-.87,.5);

\draw (2.61,.5) circle (2pt) [fill=black];
\draw (2.61,-.5) circle (2pt) [fill=black];
\draw (1.74,1) circle (2pt) [fill=black];
\draw (1.74,-1) circle (2pt) [fill=black];
\draw (2.61,-.5)--(1.74,-1);
\draw (2.61,-.5)--(2.61,.5);
\draw (1.74,1)--(2.61,.5);
\draw (1.74,1)--(.87,.5);
\draw (1.74,-1)--(.87,-.5);

\draw (3.48,-1) circle (2pt) [fill=black];
\draw (3.48,-2) circle (2pt) [fill=black];
\draw (2.61,-2.5) circle (2pt) [fill=black];
\draw (1.74,-2) circle (2pt) [fill=black];
\draw (1.74,-2)--(1.74,-1);
\draw (3.48,-2)--(2.61,-2.5);
\draw (3.48,-2)--(3.48,-1);
\draw (2.61,-.5)--(3.48,-1);
\draw (2.61,-2.5)--(1.74,-2);

\draw (2.61,-3.5) circle (2pt) [fill=black];
\draw (1.74,-4) circle (2pt) [fill=black];
\draw (.87,-2.5) circle (2pt) [fill=black];
\draw (.87,-3.5) circle (2pt) [fill=black];
\draw (.87,-2.5)--(.87,-3.5);
\draw (2.61,-3.5)--(1.74,-4);
\draw (2.61,-3.5)--(2.61,-2.5);
\draw (1.74,-2)--(.87,-2.5);
\draw (1.74,-4)--(.87,-3.5);

\draw (0,-4) circle (2pt) [fill=black];
\draw (.87,-3.5)--(0,-4);

\draw (0,0) node {$1$};
\draw (1.74,0) node {$2$};
\draw (2.61,-1.4) node {$\vdots$};
\draw (1.74,-3) node {$n$};
\end{tikzpicture}

\end{center}
\caption{Auxiliary graphs for $H_n$.} \label{auxheli}
\end{figure}

\begin{lemma} Let $z_n$ be the number of maximal matchings in $Z_n$ and $z^i_n$ be the number of maximal matchings in the auxiliary graph $Z^i_n$ in Figure \ref{auxzigzag}. Then\\

$(i)$ $h _n = h_{n-1} + h^1_{n-1} + h^2_{n-2} + h^3_{n-2}$, \\

$(ii)$ $h ^1 _n = 2h^4_{n-1} + h^5_{n-1} + h^3_{n-2} + 2h^4_{n-2} + h^5_{n-2}$, \\

$(iii)$ $h ^2 _n = h^3_{n-1} + 2h^4_{n-1} + 2h^4_{n-2} + 2h^3_{n-2} + h^5_{n-2}$, \\

$(iv)$ $h ^3 _n = h^5_n + h_n$,\\

$(v)$ $h ^4 _n = h_n + h^2_{n-1}$,\\

$(vi)$ $h ^5 _n = h^2_{n-1} + h^4_{n-1} + h^1_{n-1}$,\\

\noindent with the initial conditions $h_0=1$, $h_1=5$, $h^1_0=2$, $h^1_1=9$, $h^2_0=3$, $h^2_1=11$, $h^3_0=3$, $h^4_0=2$, and $h^5_0=2$.
\end{lemma}

\begin{theorem} Let $L(x)$, $Z(x)$, and $H(x)$ be the ordinary generating functions for the sequences $\ell_n$, $z_n$, and $h_n$, respectively. Then\\

$(i)$ 
$$
L(x) = \frac{1 + x - x^3}{1 - 4x - x^4 - x^5},
$$
$(ii)$
$$
Z(x) = \frac{1 + 2x + 4x^2 + 4x^3 + 6x^4 + 4x^5 + x^6}{1 - 3x - x^2 - 6x^3 - 7x^4 - 7x^5 - 5x^6 - x^7},
$$
$(iii)$
$$
H(x) = \frac{1 + 4x + 8x^2 + 8x^3 + 7x^4 + 4x^5 + 2x^6}{1 - x - 7x^2 - 12x^3 - 6x^4 - 7x^5 - 4x^6 - 2x^7}.
$$
\end{theorem}

\noindent Since $L(x)$, $Z(x)$, and $H(x)$ are rational functions, we can examine their denominators to obtain linear recurrences for the sequences $\ell _n$, $z_n$, and $h_n$. The initial conditions can be verified by direct computations.

\begin{corollary} \ \\

$(i)$ $\ell_n = 4\ell _{n-1} + \ell _{n-4} + \ell _{n-5}$ \\

\noindent with initial conditions $\ell _0 = 1$, $\ell _1 = 5$, $\ell _2 = 20$, $\ell _3 = 79$, and $\ell _4 = 317$,\\

$(ii)$ $z_n = 3z_{n-1} + z_{n-2} + 6z_{n-3} +7z_{n-4} + 7z_{n-5} + 5z_{n-6} + z_{n-7}$ \\

\noindent with initial conditions $z _0 = 1$, $z _1 = 5$, $z _2 = 20$, $z _3 = 75$, $z _4 = 288$, $z _5 = 1105$, and $z _6 = 4234$,\\

$(iii)$ $h_n = h_{n-1} + 7h_{n-2} + 12h_{n-3} + 6h_{n-4} + 7h_{n-5} + 4h_{n-6} + 2h_{n-7}$\\

\noindent with initial conditions $h _0 = 1$, $h _1 = 5$, $h _2 = 20$, $h _3 = 75$, $h _4 = 288$, $h _5 = 1094$, and $h _6 = 4171$.
\end{corollary}

Again we can use Darboux's Theorem to deduce the asymptotics of the sequences $\ell_n$, $z_n$, and $h_n$. The smallest modulus singularity of $L(x)$ is approximately $x=0.248804$. Hence, the asymptotic behavior of $\ell _n$ is given by $\ell_n \sim 4.01923^{n+1}$ for large $n$. Similarly, we deduce that $z_n \sim 3.83256^{n+1}$ and $h_n \sim 3.81063^{n+1}$ for large $n$.

\subsection{Extremal structure}

In this subsection, we prove the linear polyacene has most maximal matchings among all benzenoid chains of the same length.

\begin{theorem}\label{extremebenz}
Let $G_n$ be a benzenoid chain of length $n$. Then
$$
\Psi (G_n) \le \Psi (L_n).
$$
\end{theorem}

Let $G_m$ be an arbitrary benzenoid chain of length $m$. Observe that we can always draw $G_m$ as in Figure \ref{benz}, where $h_m$ is a terminal hexagon and the hexagon adjacent to the left of $h_{m-1}$ may attach at any of the edges $f, g,$ or $h$. Let us assume the hexagons of $G_m$ are labeled $h_1, \ldots, h_m$ according to their ordering in Figure \ref{benz} where ($h_1$ is the other terminal hexagon).

\begin{figure}[h]
\begin{center}
\begin{tikzpicture}[scale=1.5]

\draw (.87,-.5) circle (2pt) [fill=black];
\draw (.87,.5) circle (2pt) [fill=black];
\draw (-.87,.5) circle (2pt) [fill=black];
\draw (-.87,-.5) circle (2pt) [fill=black];
\draw (0,1) circle (2pt) [fill=black];
\draw (0,-1) circle (2pt) [fill=black];
\draw (0,1)--(.87,.5);
\draw (.87,-.5)--(.87,.5);
\draw (.87,-.5)--(0,-1);
\draw (-.87,-.5)--(0,-1);
\draw (-.87,-.5)--(-.87,.5);
\draw (0,1)--(-.87,.5);

\draw (2.61,.5) circle (2pt) [fill=black];
\draw (2.61,-.5) circle (2pt) [fill=black];
\draw (1.74,1) circle (2pt) [fill=black];
\draw (1.74,-1) circle (2pt) [fill=black];
\draw (2.61,-.5)--(1.74,-1);
\draw (2.61,-.5)--(2.61,.5);
\draw (1.74,1)--(2.61,.5);
\draw (1.74,1)--(.87,.5);
\draw (1.74,-1)--(.87,-.5);

\draw (0,0) node {$h_{m-1}$};
\draw (1.74,0) node {$h_{m}$};
\draw (-1.5,0) node {$\cdots$};

\draw (.5,.87) node {$a$};
\draw (1.2,.87) node {$b$};
\draw (.7,0) node {$c$};
\draw (1.2,-.87) node {$d$};
\draw (.5,-.87) node {$e$};
\draw (-.55,-.87) node {$f$};
\draw (-1,0) node {$g$};
\draw (-.5,.87) node {$h$};

\draw (2.2,.87) node {$x$};
\draw (2.75,0) node {$y$};
\draw (2.2,-.87) node {$z$};

\end{tikzpicture}
\end{center}
\caption{A terminal hexagon, $h_m$, and its adjacent hexagon, $h_{m-1}$, in the benzenoid chain $G_m$.} \label{benz}
\end{figure}
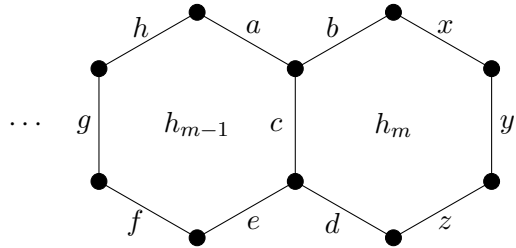

In what follows, let us adopt all of the same notation introduced in section \ref{secextremecacti}. We also make use of Lemma \ref{subgraph} introduced previously, since this holds for arbitrary graphs.

\begin{lemma} \label{edges}
Any maximal matching of $G_m$ must contain at least one of the edges $a, b, c, d$ or $e$. Moreover, any maximal matching of $G_m$ contains exactly one of these edges, or contains exactly one of the following pairs of edges: $a$ and $e$, $a$ and $d$, $b$ and $e$, or $b$ and $d$.
\end{lemma}

\begin{proof}
Take a maximal matching $M$. For sake of contradiction, suppose $M$ contains none of the edges $a, b, c, d$ or $e$. Then we could add the edge $c$ to $M$, which is a contradiction to $M$ being a maximal matching. Hence at least one of the edges $a, b, c, d$ or $e$. The remaining part of the lemma follows by considering which pairs of edges can belong to the same matching.
\end{proof}

\begin{proof}[Proof. (of theorem \ref{extremebenz})]
Take a benzenoid chain $B$ of length $n-1$. Let us set $m=n-1$ and suppose that $B$ is drawn as in Figure \ref{benz} with edges labeled as such, so that we may refer to this picture to aid this proof. We consider two cases of extending $B$ by an $n$th hexagon $h_n$.

\begin{figure}[h]
\begin{center}
\begin{tikzpicture}[scale=1.5]

\draw (.87,-.5) circle (2pt) [fill=black];
\draw (.87,.5) circle (2pt) [fill=black];
\draw (-.87,.5) circle (2pt) [fill=black];
\draw (-.87,-.5) circle (2pt) [fill=black];
\draw (0,1) circle (2pt) [fill=black];
\draw (0,-1) circle (2pt) [fill=black];
\draw (0,1)--(.87,.5);
\draw (.87,-.5)--(.87,.5);
\draw (.87,-.5)--(0,-1);
\draw (-.87,-.5)--(0,-1);
\draw (-.87,-.5)--(-.87,.5);
\draw (0,1)--(-.87,.5);

\draw (2.61,.5) circle (2pt) [fill=black];
\draw (2.61,-.5) circle (2pt) [fill=black];
\draw (1.74,1) circle (2pt) [fill=black];
\draw (1.74,-1) circle (2pt) [fill=black];
\draw (2.61,-.5)--(1.74,-1);
\draw (2.61,-.5)--(2.61,.5);
\draw (1.74,1)--(2.61,.5);
\draw (1.74,1)--(.87,.5);
\draw (1.74,-1)--(.87,-.5);

\draw (0,0) node {$h_{n-2}$};
\draw (1.74,0) node {$h_{n-1}$};
\draw (3.48,0) node {$h_{n}$};
\draw (-1.5,0) node {$\cdots$};

\draw (.5,.87) node {$a$};
\draw (1.2,.87) node {$b$};
\draw (.7,0) node {$c$};
\draw (1.2,-.87) node {$d$};
\draw (.5,-.87) node {$e$};
\draw (-.55,-.87) node {$f$};
\draw (-1,0) node {$g$};
\draw (-.5,.87) node {$h$};

\draw (2.2,.87) node {$x$};
\draw (2.75,0) node {$y$};
\draw (2.2,-.87) node {$z$};

\draw (4.35,.5) circle (2pt) [fill=black];
\draw (4.35,-.5) circle (2pt) [fill=black];
\draw (3.48,1) circle (2pt) [fill=black];
\draw (3.48,-1) circle (2pt) [fill=black];
\draw (4.35,-.5)--(3.48,-1);
\draw (4.35,-.5)--(4.35,.5);
\draw (3.48,1)--(4.35,.5);
\draw (3.48,1)--(2.61,.5);
\draw (3.48,-1)--(2.61,-.5);

\end{tikzpicture}
\end{center}
\caption{The benzenoid chain BL.} \label{BL}
\end{figure}
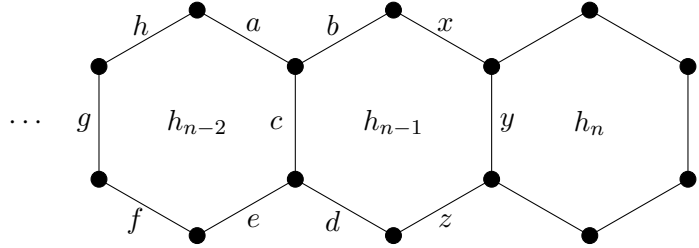

{\bf Case 1.} The hexagon $h_n$ attaches in the linear position to the edge $y$ and let us denote the resulting graph by $BL$, see Figure \ref{BL}. To compute $\Psi (BL)$ we make use of Lemma \ref{edges} and count matchings based on which of the edges $a, b, c, d, e$ are saturated. Of the possibilities in Lemma \ref{edges}, consider the maximal matchings of $BL$ containing only the edge $a$. Such a matching must also contain the edges $f$ and $z$, else this matching would contain one of the other edges $d$ or $e$. The remaining edges of the matching must be a maximal matching of $H_{n-2}\setminus \{a,f\}$ and a maximal matching of $H_{n-1,n}\setminus z$. By directly counting, we find that $\Psi (H_{n-1,n}\setminus z) = 4$ and hence, the number of maximal matchings containing only the edge $a$ is $4\cdot \Psi(H_{n-2}\setminus \{a,f\})$. We count the remaining cases from Lemma \ref{edges} similarly. We note that a $H_{n-1}\setminus c$ is used to count maximal matchings containing the edges $b$ or $d$, since these edges do not belong to the subgraph $H_{n-2}$. For example, the number of maximal matchings containing only the edge $b$ is $3\cdot \Psi(H_{n-2}\setminus \{c,f\})$. Thus 
\begin{align*}
\Psi (BL) =\ &4\cdot \Psi(H_{n-2}\setminus \{a,f\}) + 3\cdot \Psi(H_{n-2}\setminus \{c,f\})+14\cdot \Psi(H_{n-2}\setminus c) \\
&+4\cdot \Psi(H_{n-2}\setminus \{e,h\})+3\cdot \Psi(H_{n-2}\setminus \{c,h\})+9\cdot \Psi(H_{n-2}\setminus \{a,e\})\\
&+7\cdot \Psi(H_{n-2}\setminus \{a,c\})+7\cdot \Psi(H_{n-2}\setminus\{c,e\}).
\end{align*}

\begin{figure}[h]
\begin{center}
\begin{tikzpicture}[scale=1.5]

\draw (.87,-.5) circle (2pt) [fill=black];
\draw (.87,.5) circle (2pt) [fill=black];
\draw (-.87,.5) circle (2pt) [fill=black];
\draw (-.87,-.5) circle (2pt) [fill=black];
\draw (0,1) circle (2pt) [fill=black];
\draw (0,-1) circle (2pt) [fill=black];
\draw (0,1)--(.87,.5);
\draw (.87,-.5)--(.87,.5);
\draw (.87,-.5)--(0,-1);
\draw (-.87,-.5)--(0,-1);
\draw (-.87,-.5)--(-.87,.5);
\draw (0,1)--(-.87,.5);

\draw (2.61,.5) circle (2pt) [fill=black];
\draw (2.61,-.5) circle (2pt) [fill=black];
\draw (1.74,1) circle (2pt) [fill=black];
\draw (1.74,-1) circle (2pt) [fill=black];
\draw (2.61,-.5)--(1.74,-1);
\draw (2.61,-.5)--(2.61,.5);
\draw (1.74,1)--(2.61,.5);
\draw (1.74,1)--(.87,.5);
\draw (1.74,-1)--(.87,-.5);

\draw (0,0) node {$h_{n-2}$};
\draw (1.74,0) node {$h_{n-1}$};
\draw (2.61,-1.5) node {$h_{n}$};
\draw (-1.5,0) node {$\cdots$};

\draw (.5,.87) node {$a$};
\draw (1.2,.87) node {$b$};
\draw (.7,0) node {$c$};
\draw (1.2,-.87) node {$d$};
\draw (.5,-.87) node {$e$};
\draw (-.55,-.87) node {$f$};
\draw (-1,0) node {$g$};
\draw (-.5,.87) node {$h$};

\draw (2.2,.87) node {$x$};
\draw (2.75,0) node {$y$};
\draw (2.2,-.87) node {$z$};

\draw (3.48,-1) circle (2pt) [fill=black];
\draw (3.48,-2) circle (2pt) [fill=black];
\draw (2.61,-.5) circle (2pt) [fill=black];
\draw (2.61,-2.5) circle (2pt) [fill=black];
\draw (1.74,-2) circle (2pt) [fill=black];
\draw (3.48,-2)--(2.61,-2.5);
\draw (3.48,-2)--(3.48,-1);
\draw (2.61,-.5)--(3.48,-1);
\draw (1.74,-2)--(1.74,-1);
\draw (2.61,-2.5)--(1.74,-2);

\end{tikzpicture}
\end{center}
\caption{The benzenoid chain BK.} \label{BK}
\end{figure}
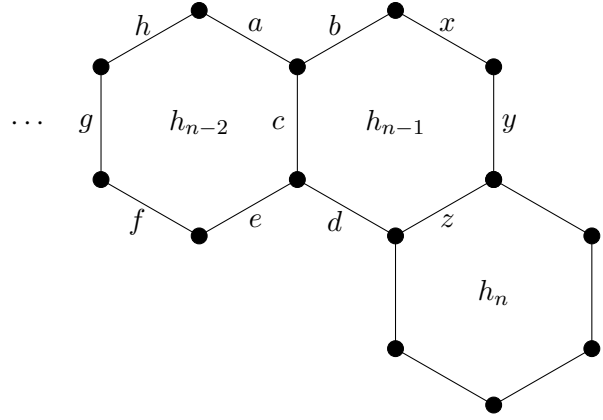

{\bf Case 2.} The hexagon $h_n$ attaches in the kinky position to the edge $z$ and let us denote the resulting graph by $BK$, see Figure \ref{BK}. Counting as in Case 1 above we obtain
\begin{align*}
\Psi (BK) =\ &6\cdot \Psi(H_{n-2}\setminus \{a,f\}) + 5\cdot \Psi(H_{n-2}\setminus \{c,f\})+12\cdot \Psi(H_{n-2}\setminus c) \\
&+5\cdot \Psi(H_{n-2}\setminus \{e,h\})+3\cdot \Psi(H_{n-2}\setminus \{c,h\})+8\cdot \Psi(H_{n-2}\setminus \{a,e\})\\
&+5\cdot \Psi(H_{n-2}\setminus \{a,c\})+7\cdot \Psi(H_{n-2}\setminus\{c,e\}).
\end{align*}

Now considering the terms in $\Psi (BL)$, by Lemma \ref{subgraph} we have 
\begin{align*}
\Psi(H_{n-2}\setminus \{a,c\}) &\ge \Psi(H_{n-2}\setminus \{a,f\}), \\
\Psi(H_{n-2}\setminus \{c\}) &\ge \Psi(H_{n-2}\setminus \{c,f\}), and\\
\Psi(H_{n-2}\setminus \{a,e\}) &\ge \Psi(H_{n-2}\setminus \{e,h\}),
\end{align*} 
implying that
\begin{align*}
\Psi (BL) \ge &6\cdot \Psi(H_{n-2}\setminus \{a,f\}) + 5\cdot \Psi(H_{n-2}\setminus \{c,f\})+12\cdot \Psi(H_{n-2}\setminus c) \\
&+5\cdot \Psi(H_{n-2}\setminus \{e,h\})+3\cdot \Psi(H_{n-2}\setminus \{c,h\})+8\cdot \Psi(H_{n-2}\setminus \{a,e\})\\
&+5\cdot \Psi(H_{n-2}\setminus \{a,c\})+7\cdot \Psi(H_{n-2}\setminus\{c,e\}) \\
&\ge \Psi (BK).
\end{align*}
The above proves that attaching a hexagon linearly gives more maximal matchings than attaching a hexagon in the kinky position. The inequality stated in the theorem follows.
\end{proof}

Again, we can see that the number of maximal matchings follows the same 
pattern as the number of independent sets, contrary to the number of all and
of perfect matchings. While the last two increase with the number of kinky
hexagons, the number of maximal matchings decreases. Further, unlike the
number of perfect matchings which does not discriminate between left and right
kinks, the number of maximal matchings seems to be sensitive to the
direction of successive turns. It seems that the helicenes have the smallest
number of maximal matchings among all benzenoid chains of the same length.

\section{Further developments}

In this last section we list some unresolved problems and indicate some
possible directions of future research. We start by stating a conjecture
about the extremal benzenoid chains.

\begin{conjecture} \label{conj1}
Let $B_n$ be a benzenoid chain of length $n$. Then $\Psi (H_n) \leq 
\Psi (B_n)$.
\end{conjecture}

Now we turn to some structural properties. The cardinality of any smallest
maximal matching in $G$ is called the {\em saturation number} of $G$. The
saturation number is of interest in the context of random sequential 
adsorption, since it gives the information on the worst possible case of
clogging the substrate; see \cite{doszub2} for a discussion and 
\cite{andova,doszub1,doslicsat} for some specific cases. However, it is not
enough to know the size of the worst possible case; it is also imprtant to
know how (un)likely is it to happen. This brings us back to enumerative 
problems, since the answer to this question depends on the ability to count
maximal matchings of a given size. A neat way to handle information about
maximal matchings of different sizes is to use the {\em maximal matching
polynomial}. It was introduced in \cite{doszub2} and some of its basic
properties were established there. There are, however, many open questions
about this polynomial. For example, for ordinary (generating) matching 
polynomials \cite{farrell79,lovasz} we know that their coefficients are 
log-concave. Is this valid also for maximal matching polynomials? We have
computed maximal matching polynomials explicitly for several families of
graphs, and we have enumerated maximal matchings in several other families.
So far, no counterexample has been found, but the proof still eludes us.

Another interesting thing to do would be to look at the dynamic aspect of
the problem, emulating the approach of Flory \cite{flory}.

Finally, it would be interesting to extend our results on other classes of
graphs, such as rotagraphs, branching polymers, composite graphs and
finite portions of various lattices.

\section*{Acknowledgements}

This work was partially supported by a SPARC Graduate Research Grant from the 
Office of the Vice President for Research at the University of South Carolina 
and also supported in part by the NSF DMS under contract 1300547. It was
supported in part by Croatian Science Foundation under
the project 8481 (BioAmpMode).


\begin{thebibliography}{10}

\bibitem{andova}
V. Andova, F. Kardo\v{s}, R. \v{S}krekovski,
\newblock Sandwiching saturation number of fullerene graphs,
\newblock {\em MATCH Commun. Math. Comput. Chem.} 73 (2015) 501--518.

\bibitem{bgw}
E. A. Bender, S. Gill Williamson,
\newblock {\em Foundations of Combinatorics with Applications},
\newblock Dover, 2006.

\bibitem{cyvingut}
S. J. Cyvin, I. Gutman,
\newblock {\em Kekul\'e Structures in Benzenoid Hydrocarbons},
\newblock Lec. Notes in Chemistry 46, Springer, Heidelberg, 1988.

\bibitem{doslicsat}
T. Do\v{s}li\'c,
\newblock Saturation number of fullerene graphs,
\newblock {\em J. Math. Chem.} 43 (2008) 647--657.


\bibitem{doslicmaloy}
T. Do\v{s}li\'c, F. Mal\o y,
Chain hexagonal cacti: Matchings and independent sets,
\newblock {\em Discrete Math.} 310 (2010) 1676--1690.

\bibitem{doszub1}
T. Do\v{s}li\'c, I. Zubac,
\newblock Saturation number of benzenoid graphs,
\newblock {\em MATCH Commun. Math. Comput. Chem.} 73 (2015) 491--500.

\bibitem{doszub2}
T. Do\v{s}li\'c, I. Zubac,
\newblock Counting maximal matchings in linear polymers,
\newblock {\em Ars Math. Contemp.}, to appear.

\bibitem{farrell79}
E. J. Farrell,
\newblock Introduction to matching polynomials,
\newblock {\em J. Comb. Theory B} 27 (1979) 75--86.


\bibitem{flory}
P. J. Flory,
\newblock Intramolecular Reaction between Neighboring Substituents of
Vinyl Polymers,
\newblock {\em J. Amer. Chem. Soc.} 61 (1939) 1518--1521.

\bibitem{frendrup}
A. Frendrup, B. Hartnell, P. D. Vestergaard,
\newblock A note on equimatchable graphs,
\newblock {\em Australas. J. Combin.} 46 (2010) 185--190.

\bibitem{montroll}
J. L. Jackson, E. W. Montroll,
\newblock Free Radical Statistics,
\newblock {\em J. Chem. Phys.} 28 (1958) 1101--1109.

\bibitem{juvan}
M. Juvan. B. Mohar, A. Graovac, S. Klav\v{z}ar, J. \v{Z}erovnik,
\newblock Fast Computation of the Wiener Index of Fasciagraphs and Rotagraphs,
\newblock {\em J. Chem. Inf. Comput. Sci.} 35 (1995) 834--840.

\bibitem{kasteleyn1}
P. W. Kasteleyn,
\newblock The statistics of dimers on a lattice. I. The number of dimer arrangements on a quadratic lattice,
\newblock {\em Physica} 27 (1961) 1209--1225.

\bibitem{kasteleyn2}
P. W. Kasteleyn,
\newblock Dimer statistics and phase transitions,
\newblock {\em J. Math. Phys.} 4 (1963) 287--293.

\bibitem{klazar}
M. Klazar,
\newblock Twelve Countings with Rooted Plane Trees,
\newblock {\em Eur. J. Comb.} 18 (1997) 195--210.

\bibitem{lovasz}
L. Lov\'asz, M. D. Plummer,
\newblock {\em Matching Theory},
\newblock North-Holland, Amsterdam, 1986.

\bibitem{oeis}
\newblock The On-Line Encyclopedia of Integer Sequences,
\newblock published electronically at {\tt http://oeis.org}

\bibitem{wagner}
S. G. Wagner,
\newblock On the number of matchings of a tree,
\newblock {\em Eur. J. Comb.} 28 (2007) 1322--1330.

\bibitem{west}
D. B. West,
\newblock {\em Introduction to Graph Theory},
\newblock Prentice Hall, Upper Saddle River, 1996.

\bibitem{wilf}
H. S. Wilf,
\newblock {\em generatingfunctionology},
\newblock Academic Press, 1990.

\end{thebibliography}

\end{document}